\documentclass[11pt,a4paper]{article}

\usepackage{amsmath,amssymb}
\usepackage{enumitem}
\usepackage{amsthm}
\usepackage{hyperref}
\usepackage{mathtools}
\usepackage{bbm}
\usepackage{graphicx,graphics,float}
\usepackage{subfigure}
\usepackage[left=1in, right=1in, top=1in, bottom=1in, includefoot, headheight=15pt]{geometry}

\newcommand\independent{\protect\mathpalette{\protect\independenT}{\perp}}
\def\independenT#1#2{\mathrel{\rlap{$#1#2$}\mkern2mu{#1#2}}}

\usepackage{color}
\definecolor{darkgreen}{rgb}{0, .5, 0}
\definecolor{darkred}{rgb}{.5, 0, 0}

\usepackage{setspace}

\theoremstyle{plain}
\newtheorem{theorem}{Theorem}[section]
\newtheorem{proposition}[theorem]{Proposition}
\newtheorem{corollary}[theorem]{Corollary} 
\newtheorem{assumption}[theorem]{Assumption} 
\newtheorem{lemma}[theorem]{Lemma} 
\newtheorem{example}[theorem]{Example}
\theoremstyle{definition} 
\newtheorem{definition}[theorem]{Definition}
\newtheorem{remark}[theorem]{Remark}

\numberwithin{equation}{section}

\usepackage{mathrsfs}

\newcommand{\E}{{\mathbb{E}}}
\newcommand{\dd}{{\rm d}}

\newcommand{\1}{\ensuremath{\mathbf{1}}}

\ifx\prop\undefined
\newtheorem{prop}{Proposition}[section]
\fi

\def\P{{\mathbb P}}

\providecommand{\abs}[1]{\ensuremath{\left\lvert#1\right\rvert}}

\providecommand{\norm}[1]{\ensuremath{\lVert#1\rVert}}
\providecommand{\var}{\ensuremath{{\rm Var}}}
\providecommand{\poi}{\ensuremath{{\rm Poi}}}

\providecommand{\cov}{\ensuremath{{\rm Cov}}}
\providecommand{\diag}{\ensuremath{{\rm Diag}}}

\begin{document}
\singlespacing %
\title{Bootstrap Percolation in Directed Inhomogeneous Random Graphs} %
\author{Nils Detering\thanks{Department of Statistics and Applied Probability, University of California, Santa Barbara, CA 93106, USA. Email: detering@pstat.ucsb.edu} \thanks{Nils Detering acknowledges financial support from the LMUexcellent initiative, the "Frankfurter Institut f\"ur Risikomanagement und Regulierung (FIRM)" and the Europlace Finance Institute. 
}\and Thilo Meyer-Brandis\thanks{Department of Mathematics, University of Munich, Theresienstra\ss e 39, 80333 Munich, Germany. Email: meyerbr@math.lmu.de} \and Konstantinos Panagiotou\thanks{Department of Mathematics, University of Munich, Theresienstra\ss e 39, 80333 Munich, Germany. Email: kpanagio@math.lmu.de} \thanks{Supported by DFG grant PA-2080/1.}%
  }
\maketitle

\begin{abstract}
Bootstrap percolation is a process that is used to model the spread of an infection on a given graph. In the model considered here each vertex is equipped with an individual threshold. As soon as the number of infected neighbors exceeds that threshold, the vertex gets infected as well and remains so forever. We perform a thorough analysis of bootstrap percolation on a novel model of directed and inhomogeneous random graphs, where the distribution of the edges is specified by assigning two distinct weights to each vertex, describing the tendency of it to receive edges from or to send edges to other vertices. Under the assumption that the limiting degree distribution of the graph is integrable we determine the typical fraction of infected vertices. Our model allows us to study a variety of settings, in particular the prominent case in which the degree distribution has an unbounded variance. Among other results, we quantify the notion of ``systemic risk'', that is, to what extent local adverse shocks can propagate to large parts of the graph through a cascade, and discover novel features that make graphs prone/resilient to initially small infections.
\end{abstract}




\section{Introduction}\label{sec:intro}

In this paper we study \emph{bootstrap percolation}, which is a classical mathematical model that is used to describe how a certain activity disperses on a given finite graph. In the classical variant of the model, one starts with a non-empty subset of the vertices, the so-called \emph{initially infected set}. The process continues in distinct rounds, and further vertices become active as soon as they have at least a certain fixed number $c \in \mathbb{N}$ of infected neighbors. After a finite number of steps the set of infected vertices will eventually stabilize, and the important question is to quantify its shape as a function of the underlying graph and the initially infected set.

The study of bootstrap percolation has a rather long history, beginning with its invention in 1979 in~\cite{Chalupa1979}, where it was used to investigate the demagnetisation properties of certain crystals. Since then, many important properties of it were studied in a broad variety of different settings, including for example the case where the underlying graph is the $d$-dimensional finite grid $[n]^d$, see \cite{MR1961342,MR2223042,MR2546747} and \cite{Balogh2012}, the extensive study for Erdos-R\'enyi random graphs~\cite{janson2012} and random regular graphs~\cite{MR2283230}, and the cases of tori~\cite{MR2214907} and infinite trees~\cite{MR2430783,MR2248323}.  From today's perspective, however, the underlying graphs that we wish to study are more complex and heterogeneous, and the details of the infection process are more intricate. Let us mention  two characteristic examples that will motivate the definition of our model:
\begin{itemize}
	\item \emph{Financial networks}. The vertices are financial institutions (like banks or insurance companies) and the edges describe monetary dependencies between them, for example a loan from one bank to another. If some institutions go bankrupt, then this may result in a cascade of credit defaults, depending on how much each remaining institution can withstand.
	\item \emph{Social networks}.  The vertices are individuals, who exchange information through announcing messages; in \emph{Twitter} for example, the users may broadcast a message to all of their followers. They, in turn, can broadcast it further, resulting again in a cascade of message transmitions.
\end{itemize}
The graphs in the two previous examples, as well as many others that appear in a variety of similar contexts, have three relevant characteristics. First, they are \emph{heterogeneous}, in the sense that the degree distribution (i.e., the probability that a uniformly random vertex has a given number of neighbors) is far from uniform -- it typically has a heavy tail. This has been verified empirically in a vast number of studies~\cite{Albert2002,Cont2013,Albert1999}. Second, the graphs are \emph{directed}, meaning that the induced relation among the vertices is not necessarily symmetric. Finally, each vertex has an \emph{individual threshold level} according to which it becomes infected; some vertices are more sensitive to activity in their neighborhood than others.

The model that we propose and study encompasses these characteristics. It contains two basic ingredients: a model for random directed graphs, where the degree distribution regarding both incoming and outgoing edges can be prescribed, and a model for bootstrap percolation, where each vertex has its own individual infection threshold. The main results of this paper include a detailed study with respect to all parameters.

Our graph model is defined as follows. Let $n \in \mathbb{N}$. Each vertex $i \in [n] := \{1, \dots, n\}$ is associated with three parameters $w_i^-, w_i^+ \in \mathbb{R}_+$ and  $c_i \in \mathbb{N}_0$. The first parameter $w_i^-$ quantifies the tendency of $i$ to receive edges from other vertices, and similarly $w_i^+$ quantifies its tendency to connect to other vertices. In particular, the probability that the \emph{directed} edge $(i,j)$, where $i \neq j$, is in the graph, is given by $\min\{1, w_i^+ w_j^-/n\}$. Moreover, all these events are assumed to be independent. This model is a generalization of the popular Chung-Lu model~\cite{ChungLu2002}, see also~\cite{Norros2006} and the extensive study~\cite{Bollobas2007}, to the setting of directed graphs.  As an auxiliary result, we show that if the joint empirical distribution of the weight sequences $(w^-_i, w_i^+)_{i \in [n]}$ converges to the distribution of an integrable random variable $(W^-,W^+)$, then the resulting in-degree and out-degree sequences are close to a bivariate mixed Poisson distribution with mixing variable $(W^-,W^+)$. Random directed graphs that were proposed prior to our work are based on the configuration model, where the \emph{actual} in- and out degrees are specified for each vertex, see e.g.~\cite{Cont2013b,Cooper2004}. These models are quite powerful and sufficient in many situations; however, they generate simple graphs (i.e., with no loops and multiple edges) with a probability that is bounded away from zero only if the degree sequence fulfills a second moment condition, see e.g.~\cite{Janson2009}. In our intended applications this condition is only rarely satisfied. The degree sequences that are frequently observed in real-world networks are so heavy tailed that they may have a variance that grows with the number $n$ of vertices, the most prominent case being a power-law distribution with exponent $2<\beta<3$~\cite{Albert1999,Cont2013}. 

The third set of parameters in our model, i.e., the quantities $(c_i)_{i \in [n]}$, describe the sensitiveness of the vertices with respect to activity in their neghborhoods. Our bootstrap percolation process on a given graph $G$ with vertex set $[n]$ is a deterministic procedure that works as follows. There is an initially infected set of vertices ${\cal A}_0 $. For a vertex $i$ let $N_G^-(i) = \{j \in [n]~|~ (j,i) \text{ is an edge of } G\}$ be the in-neighborhood of $i$ in $G$. In the $k$-th generation, where $k\in \mathbb{N}$, the infection spreads to  
\[
	{\cal A}_k = \big\{i \in [n] \mid |N_G^-(i) \cap {\cal A}_{k-1}| \ge c_i\big\}.
\]
That is, as soon as there are $c_i$ infected in-neighbors of $i$, that vertex gets infected as well and remains so forever. A straightforward consequence of this definition is that the sequence ${\cal A}_0, {\cal A}_1, {\cal A}_2, \dots$ stabilizes after at most $n-1$ generations. We say that ${\cal A}_n$ is the set of \emph{finally infected vertices}. The main result of this paper, Theorem~\ref{thm:2}, establishes in several relevant cases the typical size of ${\cal A}_n$ if the underlying graph is a directed inhomogeneous random graph in the setting described previously, and if the \emph{joint} empirical distribution of $(w_i^-, w_i^+, c_i)_{i \in [n]}$ converges to the distribution of a random vector $(W^-, W^+, C)$, where, as before, we assume that $(W^-, W^+)$ is integrable. Our setting is quite general, and it allows in particular for intricate correlation structures among the in-/out-degrees of the vertices and the infection threshold, which are expected to exist in  many natural models. For example, in specific settings it is certainly expected that vertices with high degrees have also a higher threshold, as might be the case in a financial network. Related albeit more restricted settings with varying infection thresholds have been considered before, in particular in the notable works \cite{Lelarge2012,ar:a10,Cont2013b}, where the underlying model for the graph is the configuration model. 

Based on our results we manage to quantify the notion of ``systemic risk'', that is, to which extent a small number of initial infections can propagate through the percolation process to a significant fraction of all vertices. In our setting this corresponds (informally) to the case $\P(C = 0) \to 0$, which essentially prescribes that only a sublinear fraction of vertices is initially infected.  This is a particularly important situation, as in typical applications cascades are triggered by only a small number of vertices in the network -- in a financial crisis, for example, a small number of market participants defaults initially, but this may have a severe effect on a huge part of the network. Among other results, we determine under which conditions on $(W^-, W^+, C)$ the finally infected set contains a (large) \emph{positive} fraction of all vertices that is \emph{independent} of the probability of being initially infected, as long as this is $> 0$; see Sections~\ref{sec:resapp} and~\ref{sec:apps}. Results of this type have not been shown for any other model in previous works.

\paragraph{Paper \& Proof Outline} The paper is structured as follows. In the next section we formulate in more detail our model and state the main results as well as some applications, particularly in the context of quantifying systemic risk. In Section~\ref{sec:model} we study some basic properties of the proposed directed inhomogeneous random graph. The proofs of our main results start in Section~\ref{sec:finitary}, where we first study a special case of our model. Namely,  we consider the setting where the weights $w_i^-, w_i^+$ and threshold value $c_i$ of each vertex may obtain a value from a \emph{finite} set only. Moreover, there we reformulate the activation process in a sequential form such that at each time-step only the infection from one vertex is considered. This reformulation allows us to approximate the dynamics of the system with differential equations, using the method in \cite{wormald1995}, to derive a law of large numbers for the activation process in the sequential description. The use of this method is quite common for such problems, see e.g.~\cite{Amini2014}; however, in our context the application is both conceptually and technically complex due to the three-dimensional nature of our parameter space. Moreover, since the functions defining the differential equations are only Lipchitz continuous on a domain smaller than the one of interest, we develop a novel probabilistic argument to show that the infections outside the considered domain are negligible. In Section \ref{sec:approximation}, we extend our results to the general setting by developing several couplings of the original vertex sequence to tailor-made  sequences with finitely many values; this is the main technical contribution of this paper. Difficulties arise here due to the multi-dimensionality of the vector $(W^-,W^+,C)$ and the fact that it is eventually unbounded. Based on these results, in Section~\ref{sec:apps} we prove our results on systemic risk and provide several examples. Finally, in Section~\ref{extension:relevance} we present some extensions of our main results.

\section{Results \& Applications}\label{sec:resapp}

\paragraph{Random Graph Model} For each $n \in \mathbb{N}$ we consider the vertex set $[n] =\{1,\dots , n\}$ and the set of directed edges $E:=\{(i,j) \mid i,j \in [n], i\neq j \}$. Let $\Omega:= \{0,1\}^{\abs{E}}$ and $\mathcal{F}:=2^\Omega$. We define a probability measure $\mathbb{P}$ on $(\Omega,\mathcal{F})$ in the following way. To each vertex $i\in [n]$ we assign two deterministic weights $w^-_i(n)$ and $w^+_i(n)\in \mathbb{R}_+$ and define the probability 
$p_{i,j}=p_{i,j}(n)$ for $i\neq j$ that there is a directed edge from vertex $i$ to vertex $j$ by
\begin{equation}\label{conn:prob}
p_{i,j}=\min \left\{1,{w^+_i w^-_j}/{n} \right\}.
\end{equation}
Furthermore, we assume that the events that an edge is present happens independent of the presence of all other edges.   The role of $w^-_i$ respectively $w^+_i$ is to determine the tendency of vertex $i\in [n]$ to have incoming respectively outgoing edges. Let further $({\bf w^-})(n)=(w^-_1 (n),\dots,w^-_n(n))$ and $({\bf w^+})(n)=( w_1^+ (n),\dots,w_n^+ (n))$ be the in- and out weight sequences. Observe that all the quantities including $\P, \Omega $ and $\mathcal{F}$ depend on $n$. However, to simplify notation we often neglect to mention explicitly this dependency. We denote the resulting random graph by $G_n ({\bf w}^-(n),{\bf w}^+(n))$ and we abbreviate it with $G_n ({\bf w}^-,{\bf w}^+)$. 

For a pair of in- and out-weight sequences $({\bf w^-}, {\bf w^+})$ we define their empirical distribution 
\begin{equation}
F_n(x,y)=n^{-1}\sum_{i \in [n]} \1 \{w^-_i(n) \leq x, w^+_i(n) \leq y  \}, \;\; \forall x,y \in [0,\infty ).
\end{equation} 
Let in the following $(W^-_n,W^+_n)$ be a random vector with distribution function $F_n(x,y)$. We shall pose some mild assumptions on the weight sequences.
\begin{definition}[\bf Regular Weight Sequence]\label{weight:assump}
We say that the sequence $({\bf w^-}(n), {\bf w^+}(n))_{n \ge 1}$ of pairs of weight sequences is \emph{regular}, if it satisfies the following conditions:
\begin{enumerate}
\item {\bf Convergence of weights}: There exists a distribution function $F : [0,\infty)\times [0,\infty) \rightarrow [0,1]$ such that for all  $(x,y)$ where $F$ is continuous, $\lim_{n\rightarrow \infty }F_n(x,y)=F(x,y)$.
\item {\bf Convergence of average weights:} Let $(W^-,W^+)$ be a random variable with distribution  $F$. Then $\lim_{n\rightarrow \infty }\E[(W^-_n,W^+_n)]=\E[(W^-, W^+)]=(\lambda^-,\lambda^+)$ for some $\lambda^-,\lambda^+ \in \mathbb{R}_+$.
\item {\bf Existence of a Lower Bound:} There is a $w_0 > 0$ such that $w_i^-(n), w_i^+(n) \ge w_0$ for all $n \in \mathbb{N}$ and $i \in [n]$.
\end{enumerate}
\end{definition}
\noindent 
Note that we do not require the convergence of higher moments for the vertex weights. Moreover, it can easily be seen that for a regular pair of weight sequences $\max_{i\in [n]} w^+_i=o(n)$ and $\max_{i\in [n]} w^-_i=o(n)$. We shall use this observation frequently. 

\paragraph{Bootstrap Percolation with Infection Thresholds} In addition to the weights, we assume that each vertex $i\in [n]$ is associated with an infection threshold $c_i$. The vertex $i$ becomes infected after $c_i$ of the vertices that have a directed edge to it are infected. We allow for vertices that can never be infected and assume that they have threshold $\infty$, a choice done for convenience. Set $\mathbb{N}^\infty :=\mathbb{N} \cup \{\infty\}$ and $\mathbb{N}_0^\infty :=\mathbb{N}_0 \cup \{\infty\}$. As in the case of the weights, we assume that we are given a threshold sequence ${\bf c} (n)=(c_1 (n),\dots ,c_n (n)) \in (\mathbb{N}_0^\infty)^n$. 
\begin{definition}[\bf Regular Vertex Sequence]
Let $({\bf w^-}(n), {\bf w^+}(n))_{n \ge 1}$ be a regular weight sequence and $({\bf c}(n))_{n \ge 1}$ a sequence of percolation thresholds. We call $({\bf w^-}(n), {\bf w^+}(n), {\bf c}(n)  )_{n \ge 1}$ a \emph{regular vertex sequence} if there exists a distribution function $F: \mathbb{R} \times \mathbb{R}  \times \mathbb{N}_0^\infty \rightarrow [0,1]$ such that for all points $(x,y,l) \in \mathbb{R} \times \mathbb{R}  \times \mathbb{N}_0^\infty$ for which $F(x,y,l)$ is continuous in $(x,y)$ we have $\lim_{n\rightarrow \infty} F_n(x,y,l)=F(x,y,l)$, where $F_n(x,y,l)$ is the empirical distribution function 
\begin{equation}
F_n(x,y,l)=n^{-1}\sum_{i \in [n]} \1 \{w^-_i(n) \leq x, w^+_i(n) \leq y, c_i(n)  \leq l \}, \;\; \forall (x,y,l) \in \mathbb{R} \times \mathbb{R}  \times \mathbb{N}_0^\infty.
\end{equation} 
\end{definition}
\noindent
Note that in contrast to Definition~\ref{weight:assump} of a regular weight sequence we do not pose any integrability assumptions on the threshold value.   We denote by $G_n ({\bf w}^-(n),{\bf w}^+(n), {\bf c}(n))$ the random graph $G_n ({\bf w}^-(n),{\bf w}^+(n))$ where the infection thresholds of the vertices are given by ${\bf c}(n)$, and we abbreviate it with $G_n ({\bf w}^-,{\bf w}^+, {\bf c})$. 

Given a directed graph $G$ and the threshold sequence, a bootstrap percolation process is triggered by the initial set of infected vertices $\mathcal{A}_0:= \{ i \in [n]  \mid c_i=0 \}$. Recall that for a vertex $N_G^-(i) = \{j \in [n]~|~ (j,i) \text{ is an edge of } G\}$ is the in-neighborhood of vertex $i$ in $G$. In the $k$-th generation, where $k\in \mathbb{N}$, the infection spreads to the set ${\cal A}_k$ given by
\begin{equation}
\label{eq:Sk}
	{\cal A}_k = \big\{i \in [n] \mid |N_G^-(i) \cap {\cal A}_{k-1}| \ge c_i\big\}.
\end{equation}
One can easily see that after at most $n-1$ rounds, the sequence $\mathcal{A}_0, \mathcal{A}_1, \mathcal{A}_2, \dots $ stabilizes and $\mathcal{A}_{n-1}=\mathcal{A}_n$. We call $\mathcal{A}_n$ the \emph{final set of infected vertices} in $G$.
 
\paragraph{Main result} Our main result quantifies the size of the final infected set for the random graph $G_n({\bf w^-},{\bf w^+}, {\bf c})$, where  $({\bf w^-}(n), {\bf w^+}(n), {\bf c}(n))_{n \ge 1}$ is regular. Let us introduce some notation first. For $r\in \mathbb{N}_0 \cup \{\infty\}$ let $\psi_r(x)$ denote in the rest of the paper the probability that a Poisson distributed random variable with parameter $x\geq 0$ is at least $r$, i.e.,
\begin{equation}
\label{eq:defpsi}
\psi_r(x):=\left\{ \begin{array}{ll} \P (\text{Poi}(x)\geq r)=\sum_{j\ge r} e^{-x} {x^j}/{j!} ,& r\geq 0\\
0 & r=\infty
	\end{array}.
\right.
\end{equation}
We say that a sequence of events $({\mathcal E}_n)_{n \ge 1}$ occurs \emph{with high probability} if
$\lim_{n \to \infty} \P(G_n({\bf w^-}, {\bf w^+}, {\bf c}) \in {\mathcal E}_n) = 1$. Finally, $ \xrightarrow{p}$ denotes the usual convergence in probability.
\begin{theorem}\label{thm:2}
Let $({\bf w^-}(n), {\bf w^+}(n), {\bf c}(n))_{n \ge 1}$ be a regular vertex sequence with limiting distribution $F: \mathbb{R} \times \mathbb{R}  \times \mathbb{N}^\infty_0 \rightarrow [0,1]$. Let $(W^-, W^+, C)$ be a random vector with distribution $F$. Assume $\P (C=0)>0$ and that $F(x,y,l)$ is continuous in $(x,y)$ for each $l\in \mathbb{N}$. Denote by $\hat{z}$ the smallest positive solution of 
\begin{equation}
\label{eq:fixpointallpars}
f(z;(W^-,W^+,C)):=\E [ W^+ \psi_C( W^- z ) ] -z = 0.
\end{equation}
Let ${\cal A}_n$ denote the final set of infected vertices in $G_n({\bf w^-},{\bf w^+}, {\bf c})$. Then:
\begin{enumerate}
\item For all $\epsilon>0$ with high probability $n^{-1} \abs{\mathcal{A}_n} \geq  \E [ \psi_C( W^- \hat{z} ) ]-\epsilon$.
\item \label{cond:2} If there exists $\delta > 0$ and a $\kappa < 1$ such that $\E [ W^+ W^- \P (\poi (z W^-)=C-1)\1_{C\geq 1} ] < \kappa$ for $z \in (\hat{z}-\delta, \hat{z}+\delta )$, then
\begin{equation}
n^{-1}\abs{\mathcal{A}_n} \xrightarrow{p} g(\hat{z};(W^{-},C)) :=\E [ \psi_C( W^- \hat{z} ) ], \text{ as } n\rightarrow \infty .\label{final:set:general}
\end{equation}
\end{enumerate}
\end{theorem}
Our main results provides a lower bound of essentially $g(\hat{z};(W^{-},C))$ for the fraction of vertices in the final infected set, and under the additional assumption that $\E [ W^+ W^- \P (\poi (z W^-)=C-1)\1_{C\geq 1} ]$ is less than one in a neighborhood of $\hat z$ we obtain the convergence in probability to that value. This extra assumption is needed, as otherwise the process might show a different behavior; such an effect was studied in~\cite{janson2012} for the case of Erdos-R\'enyi random graphs and infection thresholds that are the same for all vertices. We do not consider this case here.

\begin{remark}
In \cite{Cont2013b} a similar result was derived in the case where the underlying graph is generated with the configuration model and the additional assumption that the degree distribution has a bounded variance.  There, for convergence in probability it is required that the fixpoint of the relevant functional has negative derivative, i.e., that $f' (\hat{z})\leq \kappa $ for some $\kappa <0 $. If $\E [ W^+ W^- ]=\infty$ however, $f$ is not necessarily differentiable. Furthermore, it is shown in Appendix~\ref{deriv:prop:f} that if $f$ is differentiable and $f' (\hat{z})\leq \kappa $ this in fact implies that $\E [ W^+ W^- \P (\poi (z W^-)=C-1)\1_{C\geq 1} ]$ is bounded away from one in a neighborhood of $\hat z$. Thus, the formulation here is more general and covers more settings.
\end{remark}

\begin{remark}
In \cite{Amini2014} bootstrap percolation on Chung-Lu random graphs with fixed thresholds was studied, and the resulting formulas, in particular for the determination of the fixpoint~\eqref{eq:fixpointallpars} and the final fraction~\eqref{final:set:general}, have some similarity. In particular, the formulas in \cite{Amini2014} depend only on the (single) distribution of the weights and the fixed threshold, while in the present setting there naturally appears a dependency on all weight and threshold distributions. Our proof also shares some similarities with the one in \cite{Amini2014}, especially in the beginning, where we study the process on finitary weight sequences. However, the analysis of the general setting is much more involved and requires due to the multidimensionality of the parameters several new tools and ideas that are developed in Section~\ref{sec:approximation}.
\end{remark}

In the light of many applications, not only the number of finally infected vertices is of interest, but also related quantities. For example one might be interested in
\begin{equation*}
n^{-1}\sum_{i \in \mathcal{A}_n} \1 \{w^-_i(n), w^+_i(n) \geq w \},
\end{equation*}
the number of {\em large} (with weight $\ge w$) vertices that get infected. Alternatively, each vertex might have a certain {\em relevance} for the system, for example by providing some service to the system itself or to the outside. Let $r_i\geq 0$ be the relevance of vertex $i \in [n]$. If the relevance varies dramatically across vertices, the number of finally infected vertices might not be a good measure to consider. We should consider instead
\begin{equation}\label{rel:loss:rel}
(\sum_{i \in \mathcal{A}_n} r_i) \big/ (\sum_{i \in [n]} r_i),
\end{equation}
the fraction of the absolute relevance lost due to the percolation process. In \cite{Detering2015}, where a financial network is studied, the number $r_i$ comprises several properties that make a financial institution relevant to society, as for example the amount of debt issued to the real economy or its contribution to the infrastructure of the payment system. In Section~\ref{extension:relevance} we demonstrate how our results can be extended to that setting.

\paragraph{Applications}

In our main application we investigate under which conditions even a very small set of infected vertices can cause a large fraction of infected vertices at the end of the process. Let as before $({\bf w^-}, {\bf w^+}, \bf c  )$ be regular with limiting distribution $F$ and $(W^-,W^+,C)$ a random variable with distribution $F$.

Our model for studying the effect of very small initially infected sets is as follows. We assume that $\P (C=0)=0$, that is, (asymptotically) there are no initial infections. Moreover, we assume that some vertices $i \in [n]$ are being infected \emph{ex post}. In this process all vertices $i \in [n]$ receive a binary mark $m_i$, which is either $1$ or $0$, where $1$ means that the vertex keeps its initial infection threshold and $0$ that it becomes infected. Let ${\bf m}$ be the sequence of marks. We define the function $\bar{F}_n(x,y,l,m):\mathbb{R} \times \mathbb{R}  \times \mathbb{N}_0^\infty \times \{0,1\} \rightarrow [0,1]$ by 
\begin{eqnarray*}
\bar{F}_n(x,y,l,k) &= & n^{-1}\sum_{i \in [n]} \1 \{w^-_i(n) \leq x, w^+_i(n) \leq y, c_i(n)  \leq l , m_i(n)  \leq k \}
\end{eqnarray*}
and assume in the rest of this section  that $\lim_{n \rightarrow \infty} \bar{F}_n(x,y,l,m) =\bar{F}(x,y,l,m)$ for each $(x,y,l,m)$ and some distribution function $\bar{F}$. Let $(W^-,W^+,C,M)$ be a random vector distributed according to $\bar{F}$.

The following proposition investigates under which condition the fraction of infected vertices at the end of the process can be bounded away from 0 \emph{independently} of $M$.
\begin{theorem}\label{thm:epsilon:grow}
Assume that $(W^-,W^+,C)$ is such that there exists $z_0>0$ such that for any $0< z < z_0$
\begin{equation}\label{cond:small:sets}
\E [W^+ \psi_C (z W^-)] > z.
\end{equation}
Let $M$ be such that $\P (M = 0)>0$. Let $\mathcal{A}_n$ be the set of finally infected vertices in $G_n({\bf w}^-, {\bf w}^+, \overline{\bf c})$, where $\overline{c}_i(n) = c_i(n) m_i(n)$, for all $i \in \mathbb{N}$. Then with high probability
\begin{equation}\label{small:fraction:bound}
 n^{-1}\abs{\mathcal{A}_n }\geq  \E[\psi_C (W^- z_0 )] >0.
\end{equation}
\end{theorem}
Networks that fulfill the assumption of the last proposition are very prone to small initial infections, as $n^{-1}\abs{\mathcal{A}_n }$ is bounded away from zero regardless of  $\P (M = 0)$. We provide several examples. In particular, in Example~\ref{non:res:no:thres1} we show that~\eqref{cond:small:sets} often holds even if $\P (C=1)=0$, that is, when there are no \textit{weak} vertices with infection threshold equal to 1; the crucial property driving $\E [W^+ \psi_C (z W^-)]$ up is the non-existence of the second moment of the distribution of~$W^-$. This result complements our view on systemic risk  and provides a new global feature that enables us to study the vulnerability of networks.  In particular, in previous works~\cite{ar:gk10,Cont2013b} the same bootstrap percolation process was studied on random graphs that are generated according to the configuration model, where the sequence of degrees has a bounded second moment. There it was shown that an initially small infection can propagate to a big part of the graph \emph{if and only if} the subgraph induced by the so-called \emph{contagious} edges is large; an edge $(i,j)$ is contagious if $j$ is weak. Moreover, especially in the context of financial mathematics a key concept in studying the effect of initial defaults is precisely that of  contagious edges~\cite{book:sr,inc:anaa}. Our results enhance this picture; that is, our analysis reveals that in general, even if $c_i \neq 1$ for all $i\in[n]$ and $\P(C = 0) > 0$ arbitrarily small,  the initial shock can propagate to large parts of the system.

Theorem~\ref{thm:epsilon:grow} generalizes earlier research yet in another direction by not only showing the existence of a lower bound but explicitly determining it. Furthermore, we show in Section~\ref{sec:apps} that the bound is best possible and can not be improved in general. In contrast to earlier proofs of similar results using combinatorial methods, our proof of Theorem~\ref{thm:epsilon:grow} is purely analytic and provides additional insight into the role the functional $f$ plays in determining the spread of infection.

We continue by studying (essentially) all remaining cases and showing a reverse answer to Theorem~\ref{thm:epsilon:grow}, describing the situation when the network is resilient. Since the conditions looks at first sight different, see the discussion after the next theorem for the connection to Theorem~\ref{thm:epsilon:grow}.
\begin{theorem}\label{prop:resilient}
Assume that $(W^-,W^+,C)$ is such that there exists $z_0>0$ such that for any $0< z < z_0$
\begin{equation}\label{resilience}
\E [ W^+ W^- \P (\poi (z W^-)=C-1)\1_{\{C\geq 1 \}} ] <  1.
\end{equation}
Let $\{M^{(j)} \}_{j \in \mathbb{N}}$ with $\lim_{j \rightarrow \infty}\P(M^{(j)}=0)=0$ be a sequence of ex post infections, and let ${\cal A}_{n}^{(j)}$ be the set of finally infected vertices in $G_n({\bf w}^-, {\bf w}^+, {\bf c}^{(j)})$, where $c^{(j)}_i(n) = c_i(n) m_i^{(j)}(n)$, for all $i \in [n]$ and $j \in \mathbb{N}$. Then, for any $\varepsilon >0$, there exists $j_{\varepsilon}$ such that for $j \geq j_{\varepsilon}$ with high probability
\begin{equation}\label{resulzt:resilient}
n^{-1}\abs{\mathcal{A}_{n}^{(j)} } \leq \varepsilon.
\end{equation}
\end{theorem}
To see why this complements Theorem~\ref{thm:epsilon:grow}, note that $W^+ W^- \P (\poi (z W^-)=C-1)\1_{\{C\geq 1 \}}$ is non-negative. Thus, by Fubini's theorem $\E [ W^+ W^- \P (\poi (z W^-)=C-1)\1_{\{C\geq 1 \}} ]$ is the weak derivative of $\E [W^+ \psi_C (z W^-)]$, and Assumption (\ref{resilience}) implies that 
\begin{equation}\label{cond:small:f}
\E [W^+ \psi_C (z W^-)] < z
\end{equation} 
for any $0< z < z_0$, which is exactly the complement of~\eqref{cond:small:sets}.

A network for which the assumption of the theorem holds can be considered as being resilient to small infections, since the final fraction of infected vertices will still be small. In Example~\ref{example:1} we describe a family of graphs that is resilient. Condition (\ref{resilience}) is satisfied by Lemma~\ref{diff:smaller:1} if $f$ is differentiable on $(0,z_0)$ for some $z_0 > 0$ and $f'(z,(W^-,W^+,C))<0$ for $z \in (0,z_0)$.

\section{Degree Distribution in Directed Inhomogeneous Random Graphs}\label{sec:model}

In this section we study some basic properties of directed inhomogeneous random graphs, as the in- and out degrees of a single vertex and the degree sequence. Suppose that $({\bf w}^-(n), {\bf w}^+(n))_{n \ge 1}$ is a regular weight sequence with limiting distribution function $F$, and consider the sequence of graphs $G_n({\bf w}^-, {\bf w}^+)$. Denote by $X_{i,j}=X_{i,j}(n)$ the indicator function that there is a directed edge from vertex $i$ to vertex $j$. Furthermore, define the in-degree $D^-_i$ and out-degree $D^+_i$ of vertex $i \in [n]$ by
\begin{equation*}
D^-_i= \sum_{j\neq i} X_{j,i}  \quad \text{ and } \quad  D^+_i= \sum_{j\neq i} X_{i,j}.
\end{equation*}
We first compute the typical number of edges in $G_n({\bf w}^-, {\bf w}^+)$. In the following let $(W^-, W^+)$ be a random vector with distribution $F$, and set $\E[W^-] = \lambda^-$ and $\E[W^+] = \lambda^+$.
\begin{lemma}\label{nb:edges}
Denote by $e (G_n ({\bf w}^-,{\bf w}^+))$ the number of edges in $G_n ({\bf w}^-,{\bf w}^+)$. Then
\begin{equation}
{n}^{-1} e (G_n ({\bf w}^-,{\bf w}^+)) \xrightarrow{p} \lambda^- \lambda^+.
\end{equation}
\end{lemma}
\begin{proof}
We first calculate $\lim_{n\rightarrow \infty}\E [e (G_n ({\bf w}^-,{\bf w}^+))]$. Since 
$\E [X_{i,j}]=p_{i,j}$ 
\begin{equation}\label{expect:edges:1}
\E [e (G_n ({\bf w}^-,{\bf w}^+))] \leq \sum_{\substack{ i,j \in [n], i\neq j }} \frac{w_i^+ w_j^-}{n} \leq \sum_{i\in[n]} w_i^+ \sum_{j \in [n]} \frac{w_j^-}{n}.
\end{equation}
By Definition~\ref{weight:assump}, $n^{-1}\sum_{j \in [n]} {w_j^-}=\lambda^- + o(1)$ and $n^{-1}\sum_{j \in [n]} {w_j^+}=\lambda^+ + o(1)$. This implies that the right hand side of (\ref{expect:edges:1}) equals $n(\lambda^+\lambda^- + o(1))$. 
To derive a lower bound for $\E [e (G_n ({\bf w}^-,{\bf w}^+))]$, note that in order to have $w^+_i w^-_j > n$ and the minimization to $1$ in (\ref{conn:prob}) to be relevant, at least one of the two factors $w^+_i$ and $w^-_j$ has to be greater than $\sqrt{n}$. So, 
\begin{equation}\label{expect:edges:2}
\E [e (G_n ({\bf w}^-,{\bf w}^+))]\geq  n^{-1}\sum_{\substack{ i\in[n],  w_i^+ \leq \sqrt{n}}} w_i^+ \sum_{\substack{ j \in [n]\setminus \{i\}, w_i^- \leq \sqrt{n} }} {w_j^-}.
\end{equation}
By Conditions 1.\ and 2.\ of Definition~\ref{weight:assump} it follows readily that 
\begin{equation*}
\lim_{n \rightarrow \infty } n^{-1}\sum_{ w_j^- > \sqrt{n} } {w_j^-}=\lim_{n \rightarrow \infty }\E [W_n^- \1_{W_n^-> \sqrt{n} }]=0.
\end{equation*}
This, together with the same argument for the sum involving the $w_i^+$ and the fact that $\max_{i\in [n]} w^-_i=o(n)$ shows that the right hand side of (\ref{expect:edges:2}) equals $n(\lambda^+\lambda^- + o(1))$ and therefore
 \begin{equation*}
 \lim_{n\rightarrow \infty}n^{-1}{\E [e (G_n ({\bf w}^-,{\bf w}^+))]}=\lambda^+\lambda^-.
 \end{equation*} 
Since $e (G_n ({\bf w}^-,{\bf w}^+))$ is the sum of independent indicator functions, it follows that $$\var (e (G_n ({\bf w}^-,{\bf w}^+)))\leq \E[e (G_n ({\bf w}^-,{\bf w}^+))]$$ and applying the second moment method establishes the claim.
\end{proof}
Before stating the next theorem we need the following definition (see \cite{grandell1997mixed} for a treatment of uni-variate mixed Poisson distributions).
\begin{definition}{\bf Multivariate mixed Poisson distribution:}
A vector $X=(X_1,\dots,X_n)$ of random variables has a \emph{mixed Poisson distribution} with mixing distribution $F_Y$, if for every $k=(k_1,...,k_n)\in \mathbb{N}_0^n$,
\begin{equation}
\P(X=k)=\E\left[\prod_{1\leq i \leq n} e^{-Y_i} \frac{Y_i^{k_i}}{k_i !}\right],
\end{equation}
where $Y=(Y_1,\dots,Y_n)$ is a random vector with distribution function $F_Y$. 
\end{definition}
In the following we denote by $\poi(Y)$ a random vector having a mixed Poisson distribution with mixing vector $Y$. It can be easily seen that $\E[\poi (Y_i)]=\E[Y_i]$ and $\cov (\poi (Y))=\cov (Y) + \diag (\E[Y_1],\dots, \E[Y_n])$, where $\diag (a_1,\dots, a_n)$ denotes the matrix with entries $a_1,\dots, a_n$ on the diagonal and zero elsewhere.
The following Theorem~\ref{sec2:corr:1} and Lemma~\ref{sec2:degree:seq} are directed versions of known results for undirected inhomogeneous random graphs (see \cite[Thm. 6.7., Cor. 6.9]{Hofstad2014} or \cite[Thm. 3.13]{Bollobas2007}). Let $P_n (k,j)$ be the random distribution function defined by
\begin{eqnarray}
P_n (k,j) &=&n^{-1} \sum_{i\in [n]} \1_{\{D^-_i=k, D^+_i=j \}},\;\;\forall k,j \in \mathbb{N}_0 .
\end{eqnarray}
\begin{theorem}\label{sec2:corr:1}
Let $p_n (k,j)$ be the probability mass function of the mixed Poisson random variable $(\text{Poi}(W^- \lambda^+),\text{Poi}(W^+\lambda^-))$ given by 
\begin{equation}
p(k,j)=\E \left[ e^{-(W^- \lambda^+ + W^+ \lambda^-)}  \frac{(W^- \lambda^+)^k (W^+ \lambda^-)^j }{k! j!}\right].
\end{equation}
Then for all $\epsilon>0$, as $n\rightarrow \infty$
\begin{equation}
\P \left( \sum_{k,j} \abs{p(k,j) -P_n (k,j)}>\epsilon  \right)\rightarrow 0.
\end{equation}
\end{theorem}
\begin{proof}
Using Lemma~\ref{sec2:degree:seq} below, the proof in \cite[Thm.\ 3.13]{Bollobas2007} or \cite[Thm.\ 6.10.]{Hofstad2014} can be applied with some minor changes reflecting the in- and out weights and the minimization with respect to $1$ in (\ref{conn:prob}). 
\end{proof}

\begin{lemma}\label{sec2:degree:seq}
Let $k\in [n]$. There exists a coupling $(D^-_k,Z^-_k)$ and $(D^+_k,Z^+_k)$ of $D^-_k$ and $D^+_k$ respectively, where $Z^-_k$ and $Z^+_k$ are Poisson random variables with parameters $w^-_k \lambda^+$ and $w^+_k \lambda^-$ such that
\begin{eqnarray}
\P(D^-_k\neq Z^-_k) &\leq & \left( (w_k^-)^2+ w_k^- \right) o(1)\label{sec2:in:weight:near}\\
\P(D^+_k\neq Z^+_k) &\leq & \left( (w_k^+)^2+ w_k^+ \right) o(1)\label{sec2:out:weight:near}.
\end{eqnarray}
\end{lemma}
\begin{proof}
We provide the proof for $D^-_k$ only, as the argument for $D^+_k$ is similar. Define the random variables $\overline{D}^-_k$ and $\underline{D}^-_k$ by 
\begin{equation}
\underline{D}^-_k:=\sum_{\substack{  j\in [n]\setminus \{ k \}, w_k^- w^+_j \leq n}} Be \left( \frac{w_k^- w_j^+}{n} \right) \quad \text{and}\quad  \overline{D}^-_k:=\sum_{j\in [n]} Be \left( \frac{w_k^- w_j^+}{n} \right),
\end{equation}
where $Be(x)$ denote a Bernulli distributed random variable with mean $x$. Then $\underline{D}^-_k  \preceq D^-_k  \preceq \overline{D}^-_k$, where $\preceq$ denotes stochastic ordering.
Consider Poisson random variables $\underline{V}^-_k$ and $\overline{V}^-_k$ with parameters
\begin{equation}
n^{-1}\sum_{\substack{  j\in [n] \setminus \{ k \}, w_k^- w^+_j \leq n}} {w_k^- w_j^+}  \quad\text{and}\quad  n^{-1}\sum_{j\in [n]}  {w_k^- w_j^+}
\end{equation}
respectively. Recall that $\max_{i\in [n]} w^+_i =o(n)$. Then, by Definition~\ref{weight:assump}
\begin{eqnarray}
\P(\underline{D}^-_k \neq \underline{V}^-_k)&\leq &\sum_{j \in [n]} \frac{ (w^+_j)^2 (w^-_k)^2}{n^2} \leq  (w^-_k)^2 \max_{i\in [n]} w^+_i \sum_{j\in [n]} \frac{ (w^+_j)}{n^2} = (w^-_k)^2  o(1).\label{sec:2:eqn1}
\end{eqnarray}
The same estimate holds with $\underline{D}^-_k$ and $\underline{V}^-_k$ replaced by $\overline{D}^-_k$ and $\overline{V}^-_k$. To complete the proof we will couple $\overline{V}^-_k$ and $\underline{V}^-_k$ to a Poisson random variable $Z_k^-$ with parameter $w^-_k \lambda^+$.   Define $\overline{\eta}^-_k$ and $\underline{\eta}^-_k$ by
\begin{equation}
\overline{\eta}^-_k:=w^-_k \lambda^+ - \sum_{j\in [n]}  \frac{w_k^- w_j^+}{n}      \;\;\;\text{ and }\;\;\;\underline{\eta}^-_k:=\sum_{\substack{  j\in [n] \setminus \{ k \}, w_k^- w^+_j > n}} \frac{w_k^- w_j^+}{n} + \frac{w_k^- w_k^+}{n}.
\end{equation}
Then by Definition~\ref{weight:assump} $\overline{\eta}^-_k=w^-_k o(1)$ and $\underline{\eta}^-_k=w^-_k o(1)$.
If $\overline{\eta}^-_k>0$, let $\overline Y^-_k $ be a Poisson distributed random variable with parameter $\overline{\eta}^-_k$ and we define $Z_k^-:=\overline Y^-_k +\overline V_k^-$. If $\overline{\eta}^-_k<0$, we may assume that $\overline V_k^-$ is the sum of two independent Poisson distributed random variables $Z_k^-$ and $\overline Y_k^-$ with parameters $w^-_k  \lambda^+$ and $-\overline{\eta}^-_k$. In any case $ Z_k^-$ is Poisson distributed with parameter $w^-_k \lambda^+$ and 
\begin{equation}
\P(Z^-_k\neq \overline V^-_k)=\P(\text{Poi}(\abs{\overline{\eta}^-_k})\geq 1)\leq \E [\text{Poi}(\abs{\overline{\eta}^-_k})]=\abs{\overline{\eta}^-_k},
\end{equation}
due to Markov's inequality. By a similar observation we find that $\P(\underline V^-_k\neq \overline V^-_k)\leq \underline{\eta}^-_k$. Then
\begin{eqnarray*} 
\P(D^-_k\neq Z^-_k) & \leq & \P(\underline{D}^-_k\neq Z^-_k)+ \P( \overline{D}^-_k\neq Z^-_k) \\
&\leq & \P(\underline{D}^-_k\neq \underline{V}^-_k) + \P(\underline{V}^-_k\neq \overline{V}^-_k) + 2 \P(Z^-_k\neq \overline{V}^-_k) + \P( \overline{D}^-_k\neq \overline{V}^-_k) .
\end{eqnarray*}
The claim follows by combining the considerations above.
\end{proof}

\section{Bootstrap Percolation for Finitary Vertex Type Sequences}\label{sec:finitary}
In this section we study bootstrap percolation in directed inhomogeneous random graphs with so-called finitary vertex sequences that are defined below. We extend the results later in Section~\ref{sec:approximation} by approximating the general weight sequences by finitary ones.  

\begin{definition}{\bf (Regular finitary vertex sequence)}\label{finitary:vertex}
We call a regular vertex sequence $({\bf w^-}, {\bf w^+}, \bf c  )$ \emph{finitary} if there exist positive integers $l_1, l_2,c_{\max} \in \mathbb{N}$ such that the following conditions are satisfied.
\begin{enumerate}[leftmargin=*]
 \item There exist \emph{weight levels} $0 < \tilde{w}_1^- < \tilde{w}_2^-< \dots < \tilde{w}_{l_1}^-$ and $0 < \tilde{w}_1^+ < \tilde{w}_2^+< \dots < \tilde{w}_{l_2}^+$ such that $\forall i \in [n]$, $w_i^- \in \cup_{j=1}^{l_1}  \{\tilde{w}_j^- \}$ and $w_i^+ \in \cup_{j=1}^{l_2} \{ \tilde{w}_j^+ \}$, that is, the weights take only finitely many values.
 \item $\forall i \in [n]$, either $c_i\leq c_{\max}$ or $c_i=\infty$.
\end{enumerate}
\end{definition}
\noindent
Observe that for a finitary vertex sequence there exists a partition of $[n]$ given by
 \begin{equation}
[n]=\bigcup_{\substack{1\leq j \leq l_1, 1\leq k \leq l_2\\m \in [c_{\max}] \cup \{0,\infty \} } } I_{j,k;\; m} \; ,
\end{equation}
into subsets with constant  threshold and in- and out-weights, i.e. $I_{j,k;m}:=\{i \in [n] \mid  (w^-_i,w^+_i)=(\tilde{w}^-_j,\tilde{w}^+_k), c_i=m\}$. Furthermore due to the regularity, there exist  $\gamma_{j,k;m}$ with $1=\sum \gamma_{j,k;m}$ such that $\abs{I_{j,k;m}}=\gamma_{j,k;m} n (1 + o(1))$.

Before we state the main theorem of this section we define some functions that will play a crucial role in the subsequent analysis. Let $(X,Y,Z):\Omega \rightarrow \mathbb{R}_+ \times \mathbb{R}_+ \times \mathbb{N}_0^\infty$ be a random variable defined on some probability space $\Omega$ such that $\E [Y] <\infty$. 
Define the function $f :\mathbb{R}_+ \rightarrow \mathbb{R}$ by
\begin{equation}\label{definition:f}
f(z;(X,Y,Z)):=\E [ Y \psi_Z( X z ) ] -z,
\end{equation}
where $\psi$ is as in~\eqref{eq:defpsi}. Further, define the function $g :\mathbb{R}_+ \rightarrow \mathbb{R}$ by
\begin{equation}\label{definition:g}
g(z;(X,Z)):=\E [ \psi_Z( X z ) ].
\end{equation}
We will use this functions without further reference in the rest of the paper.
We shall make use of the following simple property.
\begin{lemma}\label{f:continuous}
The functions $f(z;(X,Y,Z))$ and $g(z;(X,Z))$ are continuous. Furthermore, the equation $f(z;(X,Y,Z))=0$ has a smallest solution $\hat{z}\in [0,\E [ Y ]]$.
\end{lemma}
\begin{proof}
For continuity, first observe that $z \mapsto Y \psi_Z( X z)$ is continuous point-wise, which implies 
\begin{equation*}
\lim_{h\rightarrow 0}\abs{Y(\psi_Z( X (z+h))-\psi_Z( X z))}=0.
\end{equation*}
 Furthermore, $\abs{Y (\psi_Z( X (z+h))-\psi_Z( X z))}$ is bounded by $Y$, and by assumption $\E[Y]<\infty$. The Dominated Convergence Theorem yields that 
\begin{equation*}
\lim_{h\rightarrow 0}\abs{\E [Y \psi_Z( X (z+h))] -\E[Y \psi_Z ( X z)]}=0,
\end{equation*}
from which continuity of $f$ follows. By a similar argument, the continuity of $g$ follows.
 
For the solution $f(z;(X,Y,Z))=0$, first observe that $f(0;(X,Y,Z))\geq 0$ and  $\E [ Y \psi_Z( X z ) ]\leq \E [ Y ]$, for all $z\in \mathbb{R}$ and therefore $f(\E [ Y ];(X,Y,Z))\leq 0$, which implies that a  solution exists by the continuity of $f(z;(X,Y,Z))$. This implies that the set $D_z:=\{z\in [0,\E [ Y ] \vert f(z)=0 \}$ is not empty. Further, by the continuity of $f$, the infimum of $ D_z$ must be attained and therefore $\hat{z}:=\inf D_z$ is the smallest solution for $f(z)=0$. 
\end{proof}
 The following theorem is the special case of our main result (Theorem~\ref{thm:2}) about the size of the final set of infected vertices for finitary vertex sequences.
\begin{theorem}\label{fin:weight:res}
Let $({\bf w^-}, {\bf w^+}, {\bf c})$ be a finitary regular vertex sequence and $(W^-,W^+,C)$ with $\P (C=0)>0$ a random vector with distribution equal to the limiting distribution of $({\bf w^-}, {\bf w^+}, {\bf c})$. 
Let further $\hat{z}$ be the smallest positive solution of 
\begin{equation*}
f(z;(W^-,W^+,C)) = 0.
\end{equation*}
Let $\mathcal{A}_n$ be the final set of infected vertices in $G_n ({\bf w}^-,{\bf w}^+, {\bf c})$. 
Then the following holds:
\begin{enumerate}
\item For all $\epsilon>0$ with high probability $n^{-1}\abs{\mathcal{A}_{n}} \geq  \E [ \psi_C( W^- \hat{z} ) ]-\epsilon$.
\item If $f'(\hat{z};(W^-,W^+,C))<0$, then
\begin{equation}
n^{-1}\abs{\mathcal{A}_{n}} \xrightarrow{p} \E [ \psi_C( W^- \hat{z} ) ], \text{ as } n\rightarrow \infty \label{final:set:p}.
\end{equation}
\end{enumerate}
\end{theorem}
\begin{proof}
We first show 1. We shall determine the size of the final set of infected vertices by sequentially exposing the neighbors of all vertices that are either infected initially or become infected during the process. Informally,  in the beginning we declare all initially infected vertices as \emph{unexposed}. At each step a single unexposed vertex $i \in [n]$ is considered and its neighbors are exposed. If a neighbor $j$ of $i$  becomes infected due to the new edge that is sent from $i$, it is added to the set of unexposed infected vertices. Otherwise, the threshold value of $j$ is reduced by $1$. Finally, $i$ is removed from the set of unexposed vertices; we say that $i$ is \emph{exposed}.

To describe the process formally at each step $t\in [n] \cup \{ 0 \} $ we keep track of the following sets:
\begin{enumerate}[leftmargin=*]
\item The set $U(t)$ of unexposed vertices. We set $U(0):=\{ i \in [n] \vert c_i=0 \}$.
\item The set $\overline{U}(t)$ of newly infected vertices at step $t$ in order to update $U(t)$. We set $\overline{U}(0):=\emptyset$.
\item The sets $I_{j,k;m}^{l}(t)$ for $1\leq j\leq l_1, 1\leq k\leq l_2,0\leq l <m \leq c_{\max}$ of vertices with weight levels $\tilde{w}^-_j$ and $\tilde{w}^+_k$, percolation threshold $m$ and $l$ edges from exposed vertices. Moreover, set 
\begin{equation*}
I_{j,k;m}^{0}(0) := I_{j,k;m} 
\quad \text{and} \quad
I_{j,k;m}^{l}(0) := \emptyset \,\text{ for }\, l > 0
\end{equation*}
\end{enumerate}
At step $t\in [n]$ the sets are updated by the following procedure:
\begin{enumerate}[leftmargin=*]
\item For $1\leq j\leq l_1, 1\leq k\leq l_2,0 \leq l <m \leq c_{\max}$ set
$I_{j,k;m}^{l}(t) := \emptyset$  and $\overline{U}(t):=\emptyset$.
\item Chose a vertex $v\in U(t-1)$ uniformly at random.
\item Expose all neighbors of $v$ in 
\begin{equation*}
\bigcup_{\substack{1\leq j \leq l_1, 1\leq k \leq l_2\\0 \leq l<m \leq c_{\max}} } I_{j,k;m}^{l}(t-1).
\end{equation*}
 Let $w\in I_{j,k;m}^{l}(t-1)$. If there is no edge send from $v$ to $w$ ($X_{v,w}=0$), place $w$ in $I_{j,k;m}^{l}(t)$. If there is an edge ($X_{v,w}=1$) and $l=m-1$, place $w$ in $\overline{U}(t)$. If there is an edge ($X_{v,w}=1$) and $l<m-1$, then place $w$ in $I_{j,k;m}^{l+1}(t)$.
\item Set $U(t):= (U(t-1)\setminus \{v\}) \cup \overline{U}(t)$.
\end{enumerate}
Edges that are sent to already infected vertices are not exposed. The above steps are repeated until step $\hat{t}$, the first time $U(t)$ is empty. Note that $\hat{t}$ is the final number of infected vertices.

Denote by $c_{j,k;m}^{l}(t)$ the size of $I_{j,k;m}^{l}(t)$ for $1\leq j\leq l_1, 1\leq k\leq l_2,0 \leq l<m \leq c_{\max}$ and by $u(t)$ the size of $U(t)$. 
Additionally we need to keep track of the total out going weights in the set $U(t)$, which we shall denote by $w (t)$.
Further we use the vector $h (t)$ to describe the state of the entire system, that is 
\begin{equation*}
h (t)=\left(u(t),w (t),\{c_{j,k;m}^{l}(t)\}_{\substack{j \in [l_1], k\in [l_2]\\0 \leq l<m \leq c_{\max}} }\right).
\end{equation*}
First observe that for $n$ sufficiently large, we can ignore the minimization in (\ref{conn:prob}), since $w_i^+ w_i^-$ is bounded for finitary weight sequences and the denominator in (\ref{conn:prob}) is $n$. Conditioning on the weight of the selected vertex and using the law of total expectation one obtains that the expected evolution of the system is governed by the following equations:
\begin{eqnarray*}
\E[u(t)-u(t-1)\vert  h(t-1)]&=&-1+ n^{-1}\sum_{j} \left(\sum_{k,m} c^{m-1}_{j,k;m}(t-1)\right)\frac{\tilde{w}_j^-w (t-1)}{u(t-1)},\label{markov:chain:1}\\
\E[w (t)-w (t-1)\vert  h (t-1)]&=&- \frac{w (t-1)}{u(t-1)} 
+ n^{-1}\sum_k \tilde{w}_k^+ \left( \sum_{j,m} c^{m-1}_{j,k;m}(t-1)  \frac{\tilde{w}_j^-w (t-1)}{ u(t-1)} \right),
\end{eqnarray*}
and for the index set $1\leq j \leq l_1, 1\leq k \leq l_2,1\leq m < c_{\max}$ and $0\leq l < m$
\begin{eqnarray*}
\E[c^{l}_{j,k;m}(t)-c^{l}_{j,k;m}(t-1)\vert  h(t-1)] = n^{-1}\left(\1_{\{l\neq 0\}}c^{l-1}_{j,k;m}(t-1) - c^{l}_{j,k;m}(t-1)\right) \left(\frac{\tilde{w}_j^- w (t-1)}{u(t-1)}\right).
\end{eqnarray*}
We will approximate the quantities $h (t)/n$ using the method proposed in \cite{wormald1995} by a vector valued function 
\begin{equation}
\left(\nu (\tau) ,\mu (\tau),\{\gamma^{l}_{j,k;m}(\tau)\}_{\substack{1\leq j \leq l_1, 1\leq k \leq l_2\\0 \leq l<m \leq c_{\max}} }\right)
\end{equation}
 solving the following system of ordinary differential equations:
\begin{eqnarray}
\frac{\dd \nu (\tau)}{\dd \tau}&=&-1 + \sum_{j} \left(\sum_{k,m} \gamma^{m-1}_{j,k;m}(\tau)\right)\frac{\tilde{w}_j^-\mu (\tau)}{ \nu(\tau)},\label{ode:system:2}\\
\frac{\dd \mu (\tau)}{\dd \tau}&=&- \frac{\mu (\tau)}{\nu (\tau)} + \sum_k \tilde{w}_k^+ \left( \sum_{j,m} \gamma^{m-1}_{j,k;m}(\tau)  \frac{\tilde{w}_j^- \mu (\tau)}{\nu (\tau)}\right),\label{ode:system:3}\\
\frac{\dd \gamma^{l}_{j,k;m}(\tau)}{\dd \tau}&=&\left(\1_{\{l\neq 0\}}\gamma^{l-1}_{j,k;m}(\tau) - \gamma^{l}_{j,k;m}(\tau)\right) \left(\frac{\tilde{w}_j^- \mu (\tau)}{\nu (\tau)}\right),\label{ode:system:1}
\end{eqnarray}
with initial conditions
\begin{eqnarray}
\nu (0) &=& \mathbb{P} (C=0),\label{ode:system:ini:1} \\
\mu (0)&=&\sum_{k} \tilde{w}_k^+ \P(W^+=\tilde{w}^+_k,C=0), \label{ode:system:ini:2} \\ 
\gamma^{0}_{j,k;m}(0)&=& \P(W^-=\tilde{w}^-_j,W^+=\tilde{w}^+_k,C=m), \label{ode:system:ini:3} \\
\gamma^{l}_{j,k;m}(0)&=&0, \text{ for } 0<l<m \label{ode:system:ini:4}.
\end{eqnarray}
For $\delta_1,\delta_2>0$ we consider the domain
\begin{eqnarray}\label{convergence:area}
D_{\delta_1,\delta_2}
	=
	\Big\{
		(\tau ,\nu, \mu,\gamma_{j,k;m}^{l}) \in \mathbb{R}^{b+1}
		&\mid & -\delta_1 < \tau < 1,
		-\delta_1 < \frac{\mu}{\nu} <  2 \tilde{w}_{l_2}^+ ,
		-\delta_1 < \gamma_{j,k;m}^{l} < \gamma_{j,k;m} + \delta_1, \nonumber \\ 
		&& \delta_2 < \nu <1 + \delta_1, 0 < \mu < 2 \tilde{w}_{l_2}^+ \Big\}\label{convergenceset}.
\end{eqnarray}
The system (\ref{ode:system:1})-(\ref{ode:system:2}) fulfills a Lipschitz condition on $D_{\delta_1,\delta_2}$ for $\delta_2>0$ as can be easily seen by calculating the partial derivatives.
Further, in order to apply \cite[Thm. 1]{wormald1995} we have to show that there are functions $\lambda(n)=\lambda$ with $\lambda \rightarrow \infty$ and $\omega (n)=\omega $ such that $\lambda^4 \log n < \omega < n^{2/3} /\lambda $ and
\begin{equation}\label{bounds:diff:lambda}
\mathbb{P} \left( \norm{h (t+1) - h (t)}_{\max} > \frac{\sqrt{\omega}}{\lambda^2 \sqrt{ \log n}} \vert h(t)  \right)=o(n^{-3}).
\end{equation}
If one chooses $\lambda (n)=n^{1/8}$ and $\omega (n) = B^2 n^{25/48}$ where $B>0$ is a constant, then 
\begin{equation}
\frac{\sqrt{\omega}}{\lambda^2 \sqrt{\log n}}=B n^{1/96}/\sqrt{\log n},
\end{equation}
and it remains to show that the maximal degree is bounded by $n^{1/96}/\sqrt{\log n}$, since due to the bounded weights, $B$ can be chosen such that (\ref{bounds:diff:lambda}) holds. This is done in Lemma~\ref{max:degrees:sum}.
According to \cite[Thm. 1]{wormald1995} we get
\begin{eqnarray}\label{dgl:convergence}
c^{l}_{j,k;m}(t)/n&=& \gamma^{l}_{j,k;m}(t/n) + o_p (1) \label{convergence:1}\\
u(t)/n&=& \nu (t/n) + o_p (1) \label{convergence:2}\\
w(t)/n&=& \mu (t/n) +o_p (1)\label{convergence:3}
\end{eqnarray}
where $\gamma^{l}_{j,k;m}(\tau),\nu (\tau)$ and $\mu (\tau)$ solve (\ref{ode:system:1})-(\ref{ode:system:3}) and where (\ref{convergence:1})-(\ref{convergence:3}) holds until the solution leaves $D_{\delta_1,\delta_2}$. Since $\delta_2$ can be chosen arbitrarily close to $0$, it is clear that the solution can be extended to the region $D_{\delta_1,0}$.

An easy but tedious calculation shows that the solutions $\nu$ and $\mu$ are given by
\begin{equation}
\nu(\tau) =\nu (0) -\tau +\sum_{j,k}  \left(\sum_{m} \gamma^{0}_{j,k;m}(0) \P \left[ \text{Poi} \left( \tilde{w}_j^- z(\tau) \right) \geq m \right] \right)\label{solution:nu}
\end{equation}
and
\begin{equation}\label{solution:mu}
\mu (\tau) =\mu (0) -\int_0^{\tau} \frac{\mu (s)}{\nu (s)}\dd s +\sum_{j,k,m}   \tilde{w}_k^+ \gamma^{0}_{j,k;m}(0) \P \left[ \text{Poi} \left( \tilde{w}_j^- z(\tau) \right) \geq m \right],
\end{equation}
 with $z(\tau):=\left( \int_0^{\tau}\frac{\mu (s)}{\nu (s)}\dd s \right)$.
Define
\begin{equation}
\tau_{D_{\delta_1,\delta_2}}=\min \{\tau \vert (\tau ,\nu (\tau), \mu (\tau),\gamma_{j,k;m}^{l} (\tau)) \notin D_{\delta_1,\delta_2}) \}
\end{equation}
Observe that $f(z(\tau))=\mu(\tau)$ for $\tau<\tau_{D_{\delta_1,0}}$. Since $z$ is strictly increasing in $\tau$ as long as $(\tau ,\nu (\tau), \mu (\tau),\gamma_{j,k;m}^{l} (\tau)) \in D_{\delta_1,0}$ the function $z$ is injective.
We need to ensure that we can choose $\delta_2$ small enough such that the process can be approximated arbitrarily close to $\hat{\tau}$, which is such that $z(\hat{\tau})$ equals $\hat{z}$, the smallest zero of $f$. Observe that $\hat{z}>0$, since $\P(C=0)>0$ implies that $f(0)>0$.

Therefore we need to show that for any given $\epsilon$ we can chose $\delta_2$ small enough such that there exists $\tau_{\epsilon}< \tau_{D_{\delta_1,\delta_2}}$ with 
\begin{equation}\label{ineq:eps}
\hat{z} - \epsilon < \int_0^{\tau_{\epsilon}}\frac{\mu (s)}{\nu (s)}\dd s.
\end{equation}
Since $\hat{z}$ is assumed to be the first zero of $f$ and since $f$ is continuous on the compact set $[0,\hat{z} - \epsilon ]$ it attains its minimum at some point $z_{\min} \in [0,\hat{z} - \epsilon ]$. 
Further, observe that we must have $\nu (\tau)\geq \mu(\tau)/\tilde{w}^+_{l_2}$ such that choosing 
\begin{equation}
0<\delta_2 (\epsilon) <f(z_{\min})/\tilde{w}^+_{l_2}
\end{equation}
ensures that there exists $\tau_{\epsilon} < \tau_{D_{\delta_1,\delta_2(\epsilon)}} $ such that the inequality in (\ref{ineq:eps}) holds and the convergence in (\ref{dgl:convergence}) holds at least until $\tau_{\epsilon}$.
Since $\epsilon$ can be chosen arbitrarily close to $0$, we can conclude that $\mu (\tau)$ converges to $0$ as $z(\tau)$ approaches $\hat{z}$ by continuity of $f$. Because $\tilde{w}^+_1 \nu(\tau)\leq \mu (\tau)\leq \tilde{w}_{l_2}^+ \nu(\tau) $ on $D_{\delta_1,0}$ we know that also $\nu(\tau)$ converges to $0$.
For any given $\epsilon$ we get that 
\begin{equation}\label{u:fraction:stop}
u(\lfloor \tau_{\epsilon} n \rfloor )/n=\nu (\tau_{\epsilon} ) + o_p(1).
\end{equation}
From $\nu (\hat{\tau})=0$ it follows that
\begin{eqnarray}
\hat{\tau} &=&\nu (0) +\sum_{j,k}  \left(\sum_{m} \gamma^{0}_{j,k;m}(0) \P \left[ \text{Poi} \left( \tilde{w}_j^- z( \hat{\tau}) \right) >m \right] \right) \nonumber\\
&=& \nu (0) +\sum_{j,k}  \left(\sum_{m} \gamma^{0}_{j,k;m}(0) \P \left[ \text{Poi} \left( \tilde{w}_j^-  \hat{z} \right) >m \right] \right) = \E [ \psi_C( W^- \hat{z} ) ].\label{tau:equality}
\end{eqnarray}
Since $\abs{\mathcal{A}_n}/n\geq\hat{\tau} + o_p(1)$, the claim follows.

In order to prove 2.\ we need to show that the process $u(t)$ becomes zero soon after the $\tau_{\epsilon} n$ steps or equivalently, that the remaining infections triggered by  $U(\lfloor \tau_{\epsilon} n\rfloor )$ are negligible. We shall expose all remaining vertices in $U$ at once and bound the number of infections triggered by $U(\lfloor \tau_{\epsilon} n\rfloor)$. 
Denote by $\mathcal{W}:=\bigcup_{ j , k , m} I_{j,k;m}^{m-1}$ the set of (weak) vertices that need only one more infected neighbor to become infected and by $\mathcal{S}:=\bigcup_{ j , k , m\geq l+2} I_{j,k;m}^{l}$ the set of (strong) vertices that need at least two more infected neighbors to become infected. Further, denote $\mathcal{N}_l \subset \mathcal{W}\cup \mathcal{S}$ the set of vertices that become infected in the $l$th round after exposing $U(\lfloor \tau_{\epsilon} n \rfloor )$ and define
\begin{eqnarray*}
\mathcal{W}_l&:=&\mathcal{W} \cap \mathcal{N}_l, \qquad
\mathcal{S}_l := \mathcal{S} \cap \mathcal{N}_l \enspace.
\end{eqnarray*}
Since $(W^-,W^+,C)$ takes only finitely many values, observe that differentiation under the integral can be justified and it readily follows that $f$ is continuously differentiable. Moreover, by~(\ref{ode:system:3}) and the fact that $f^{'}(\hat{z})=\kappa <0$ by assumption and by continuity of $f^{'}$ we can chose $\epsilon_\kappa>0$ such that for $0 < \epsilon<\epsilon_\kappa$
\begin{equation*}
\left( \sum_{j,k,m} \gamma_{j,k;m}^{m-1} (\tau_{\epsilon}) \tilde{w}^-_j \tilde{w}^+_k\right)\leq 1 + \kappa /2<1.
\end{equation*}
Further observe that $\sum_{j,k,m} \gamma_{j,k;m}^{m-1}(\tau_{\epsilon} )<1$ for $0<\epsilon\leq \epsilon_\kappa$.
Set 
\begin{equation*}
c_1:=\max \Big\{1 + \kappa /2 ,\sum_{j,k,m} \gamma_{j,k;m}^{m-1}(\tau_{\epsilon_\kappa} ) \Big\}<1
\end{equation*}
and chose  $0<c_2,c<1$ such that $0\leq c_1+c_2\leq c<1$. Further define
\begin{eqnarray*}
C_1&:=&2\max \{ \tilde{w}^-_{l_1}, 1 \}\\
C_2&:=&\max \{ \left( \tilde{w}^-_{l_1} \right)^2,(\tilde{w}^-_{l_1} \tilde{w}^+_{l_2}  )^2\} \frac{C_1^2}{1-c} \\
C_3&:=& \max \{ \tilde{w}^-_{l_1} \tilde{w}^+_{l_2},1 \} 
\end{eqnarray*}
and chose $x_0$ such that 
$C_2 x ^2\leq ((c_1 c_2 c)/C_3) x$ for $0<x< x_0$. Chose further $\epsilon_0$ such that $2 \mu (\tau_{\epsilon} ) \leq x_0$ for $\epsilon<\epsilon_0$. Let the event $\mathcal{A}$ be defined by $\mathcal{A}:=\{ w( \lfloor \tau_{\epsilon} n\rfloor )/n \leq 2 \mu (\tau_{\epsilon} )\}$ and observe that $\lim_{n \rightarrow \infty} \P (\mathcal{A}^c)=0$. We shall prove by induction on $l$ that for $\epsilon<\epsilon_0$
\begin{eqnarray}
n^{-1}\E[\abs{\mathcal{W}_l} \cdot \1_{\mathcal{A}} ]&\leq &c_1 c^{l-1} \mu (\tau_{\epsilon} )  C_1  \label{p:1} \\
n^{-1}\E[\abs{\mathcal{S}_l} \cdot \1_{\mathcal{A}} ]&\leq &C_2 \mu (\tau_{\epsilon} ) ^2 c^{l-2} \leq (c_2/C_3) c^{l-1} \mu (\tau_{\epsilon} ).\label{p:2}
\end{eqnarray}
The estimates (\ref{p:1}) and (\ref{p:2}) especially imply that 
\begin{equation*}
n^{-1}\E[ ( \abs{\mathcal{W}_l} + \abs{\mathcal{S}_l} ) \cdot \1_{\mathcal{A}}] \leq c^l C_1 \mu (\tau_{\epsilon} ) 
\end{equation*}
and 
\begin{equation}\label{size:final:set}
n^{-1}\left( \sum_l \E[ ( \abs{\mathcal{W}_l} + \abs{\mathcal{S}_l} ) \cdot \1_{\mathcal{A}} ] \right) \leq \frac{C_1}{1-c} \mu (\tau_{\epsilon} ) .
\end{equation}
For $l=1$ observe that for a vertex $x\in \mathcal{I}^{1}$ the conditional probability to become infected by vertices in $U(\lfloor \tau_{\epsilon} n\rfloor)$ given that $\mathcal{A}$ happened, is bounded by
\begin{equation*}
\tilde{w}^-_{l_1}   \mu( \tau_{\epsilon}) 2
\end{equation*}
such that
\begin{equation}
n^{-1}\E[\abs{\mathcal{W}_1} \cdot \1_{\mathcal{A}}] \leq \sum_{j,k,m} \gamma_{j,k;m}^{m-1}(\tau_{\epsilon} )  \tilde{w}^-_{l_1}  \mu( \tau_{\epsilon}) 2 \leq c_1 \mu (\tau_{\epsilon} )  C_1.\label{induc:I:1:1}
\end{equation}
Further, the conditional probability for a vertex to be in $\mathcal{S}_1$, given  $\mathcal{A}$, is bounded by
\begin{equation*}
\left( \tilde{w}^-_{l_1}  \mu( \tau_{\epsilon})2 \right)^2
\end{equation*}
and thus choosing $\epsilon<\epsilon_0$ such that $2\mu (\tau_{\epsilon} )<x_0$ yields $\E[\abs{\mathcal{S}_1}\cdot\1_{\mathcal{A}}] /n  \leq C_2 \mu (\tau_{\epsilon} )^2 \leq c_2/C_3 \mu (\tau_{\epsilon} )$ by definition of $C_2$ and $x_0$ for $n\geq n_0$.

Assume now that (\ref{p:1}) and (\ref{p:2}) hold for $1\leq k \leq l$. For a vertex $x\in \mathcal{W}$ to be in $\mathcal{W}_{l+1}$ it needs to have at least one neighbor in $\mathcal{N}_{l-1}$. We shall show the slightly stricter recursion for $\mathcal{W}_{l}$, namely
\begin{equation*}
\frac{1}{n}\left( \sum_{x\in \mathcal{W}} w^-_{x} \right) \sum_{y\in \mathcal{W}\cup \mathcal{S}} \frac{w_y^+ \mathbb{P}(y \in \mathcal{N}_{l})}{n}\leq c_1 c^{l} \mu (\tau_{\epsilon} )  C_1,
\end{equation*}
given that $\mathcal{A}$ happened, from which clearly (\ref{p:1}) follows (note that for $l=1$ this was captured in (\ref{induc:I:1:1}) already) since 
\begin{equation*}
\E[\abs{\mathcal{W}_{l+1}} \cdot \1_{\mathcal{A}} ]\leq \sum_{x \in \mathcal{W},y\in \mathcal{W}\cup \mathcal{S}} \mathbb{P} ( X_{x,y}=1  )   \mathbb{P} (y \in \mathcal{N}_{l} ).
\end{equation*}
First observe that
\begin{eqnarray}
&& \frac{1}{n}\left( \sum_{x\in \mathcal{W}} w^-_{x} \right) \sum_{y\in \mathcal{W}\cup \mathcal{S}} \frac{w_y^+ \mathbb{P}(y \in \mathcal{N}_{l})}{n}\nonumber \\
&= & \frac{1}{n} \sum_{x \in \mathcal{W},y\in \mathcal{W}  } \mathbb{P} ( X_{x,y}=1 )   \mathbb{P} (y \in \mathcal{W}_{l} ) +\frac{1}{n} \sum_{x \in \mathcal{W},y\in \mathcal{S}} \mathbb{P} ( X_{x,y}=1 )   \mathbb{P} (y \in \mathcal{S}_{l} )\nonumber \\
& \leq & \frac{1}{n} \left( \sum_{x\in \mathcal{W}} w^-_{x}\right) \sum_{y\in \mathcal{W}} \frac{w_y^+ \mathbb{P}(y \in \mathcal{W}_{l})}{n} +  \frac{1}{n} \tilde{w}^-_{l_1} \tilde{w}^+_{l_2} \E[\abs{\mathcal{S}_{l}}] \nonumber \\
&\leq & \frac{1}{n} \left( \sum_{x\in \mathcal{W}} w^-_{x} \right) \left( \sum_{z\in \mathcal{W}} \frac{w_z^+  w_z^-}{n} \right)  \sum_{y\in \mathcal{W}\cup \mathcal{S}} \frac{w_y^+ \mathbb{P}(y \in \mathcal{N}_{l-1})}{n} + \frac{1}{n} \tilde{w}^-_{l_1} \tilde{w}^+_{l_2}  \E[\abs{\mathcal{S}_{l}}].\label{induction:end}
\end{eqnarray}
The middle factor in the first summand is bounded by $c_1$ by definition of $c_1$. The induction step then implies that (\ref{induction:end}), given that $\mathcal{A}$ happened, is bounded by 
\begin{equation*}
c_1 c_1 c^{l-1} \mu (\tau_{\epsilon} )  C_1  + C_3 C_2 \mu (\tau_{\epsilon} )^2 c^{l-2}  \leq c_1 c^{l} \mu (\tau_{\epsilon} )  C_1.
\end{equation*}
To calculate $\E[\abs{\mathcal{S}_{l+1}} \cdot \1_{\mathcal{A}}]$, we first observe that for a vertex to be in $\mathcal{S}_{l+1}$ it needs to have at least one neighbor in $\mathcal{N}_l$ and one in $\cup_{k\leq l} \mathcal{N}_k$. Using $\E [\abs{\mathcal{W}_k}]=\sum_{x\in \mathcal{W}} \mathbb{P}(x \in \mathcal{W}_k)$ and $\E[\abs{\mathcal{S}_{l+1}}]=\sum_{x\in \mathcal{S}} \mathbb{P}(x \in \mathcal{S}_{l+1})$ for $k\leq l$ we find
\begin{equation*}
\E[\abs{\mathcal{S}_{l+1}}]\leq n(\tilde{w}^-_{l_1} \tilde{w}^+_{l_2} /n)^2 \sum_{x\in \mathcal{W}\cup \mathcal{S}} \mathbb{P}(x \in \mathcal{N}_{l}) \sum_{x\in \mathcal{W}\cup \mathcal{S}} \mathbb{P}(x \in \cup_{k\leq l} \mathcal{N}_{k}).
\end{equation*}
However, by the induction step we know that 
\begin{equation*}
\sum_{x\in \mathcal{W}\cup \mathcal{S}} \mathbb{P}(x \in \mathcal{N}_{l})\leq n c^l \mu (\tau_{\epsilon} ) C_1
\end{equation*}
and
\begin{equation*}
\sum_{x\in \mathcal{W}\cup \mathcal{S}} \mathbb{P}(x \in \cup_{k\leq l} \mathcal{N}_{k})\leq \sum_{k\leq l} n c^k C_1 \mu (\tau_{\epsilon} )  \leq  n \frac{C_1}{1-c}  \mu (\tau_{\epsilon} ),
\end{equation*}
given that $\mathcal{A}$ happened. This yields 
\begin{equation*}
n^{-1}\E[\abs{\mathcal{S}_{l+1}} \cdot \1_{\mathcal{A}}] \leq C_2 c^{l-1} \mu ( \tau_{\epsilon}  ) \mu ( \tau_{\epsilon}  ) \leq (c_2/C_3) c^{l} \mu ( \tau_{\epsilon}  ),
\end{equation*}
proving (\ref{p:2}).
By (\ref{size:final:set}) and Markov's inequality we get
\begin{equation}\label{markov:ineq}
\mathbb{P} \left( n^{-1} \sum_{1\leq l\leq n} (\abs{\mathcal{W}_l} + \abs{\mathcal{S}_l})\cdot \1_{\mathcal{A}} \geq \sqrt{\frac{C_1}{1-c} \mu (\tau_{\epsilon} )} \right) \leq \sqrt{\frac{C_1}{1-c} \mu (\tau_{\epsilon} )}.
\end{equation}
Since $\abs{\mathcal{A}_n}/n\leq \hat{\tau} + \left(\sum_l \abs{\mathcal{W}_l} + \abs{\mathcal{S}_l}  \right)/n$ and $\lim_{n \rightarrow \infty} \P (\mathcal{A}^c)=0$, the claim in (\ref{final:set:p}) follows from (\ref{markov:ineq}) together with (\ref{tau:equality}).
\end{proof}
In the proof of the last theorem we needed the following simple (and far from tight) result about the maximum degree in the random graph.
\begin{lemma}\label{max:degrees:sum}
Let $D^- :=\max_{i\in [n]} D^-_i$ and $D^+:=\max_{i\in [n]} D^+_i$ denote the maximal in- and out-degrees in $G_n({\bf w}^-,{\bf w}^+)$, where $({\bf w}^-,{\bf w}^+)$ is finitary.
Then $\P (D^+,D^- \geq  n^{1/100})=o(n^{-3})$.
\end{lemma}
\begin{proof}
We shall prove the bound for $D^-$, the bound for $D^+$ is analogue. Let $w:=\max \{\tilde{w}^-_{l_1} , \tilde{w}^+_{l_2} \}$ and observe that
\begin{equation*}
\mathbb{P} (D^-\geq k)\leq n \binom{n-1}{k} \left( \frac{w^2 }{n} \right)^{k} \leq  n\frac{w^{2k}}{k\!!}.
\end{equation*}
The proof is completed by noting that for $k\ge n^{1/100}$ and large $n$ we have $k! \geq n^5 w^{2k}$.
\end{proof}

\section{General Vertex Sequences}\label{sec:approximation}
In this section we prove Theorem~\ref{thm:2}, that is, we show that Theorem~\ref{fin:weight:res} extends to non-finitary vertex sequences. We use the results of the last section to approximate  general vertex sequences by two finitary ones in a tailor-made way. The approximation is such that one sequence ultimately generates a graph that gives a lower bound for the final fraction of infected vertices, while the second sequence generates a graph that provides an upper bound. A sandwich type argument then allows us to determine $|\mathcal{A}_n|$  in Section~\ref{proof:thm:2}. For this argument we first show that the functions defined in (\ref{definition:f}) and (\ref{definition:g}) depend continuously on the random variables involved. For bounded domains, the necessary results are provided by Helly's theorem, which states that if a sequence of distributions $F^i$ converges point-wise to $F$ and the function $h$ is continuous and bounded on some bounded domain $D$, then $\int_{D} h \dd F^i \rightarrow \int_{D} h \dd F$. However, to investigate convergence of the functions defined in (\ref{definition:f}) and (\ref{definition:g}) we are faced with an unbounded domain and unbounded support. Therefore Helly's theorem cannot be applied directly and we use a tailor-made approximation of $F$ that ensures that the integrals over several functions relevant in the following analysis are convergent.

Let $(W^{-},W^{+},C)$ be a random vector with distribution function $F$ fulfilling the properties of Theorem~\ref{thm:2}. To avoid confusion we use the expectation operator $\E$ only with respect to the measure defined by $F$ on $ \mathbb{R}_+ \times \mathbb{R}_+  \times \mathbb{N}_0^\infty$. For the approximating measures we use the integral notation. For the approximation we can restrict to the set $\mathbb{R}_+ \times \mathbb{R}_+  \times \mathbb{N}_0\subset \mathbb{R}_+ \times \mathbb{R}_+  \times \mathbb{N}_0^\infty$, since all involved functions have $\psi_r(x)$ as a factor, which is zero for $r=\infty$. Define the sets
\begin{equation}
\label{eq:Di}
D_{i}:=\{(x,y,l) \in \mathbb{R}_{+} \times \mathbb{R}_{+} \times \mathbb{N}_0 \mid  x, y, l\leq i\},D_{\infty}:= \mathbb{R}_+ \times \mathbb{R}_+ \times \mathbb{N}_0 \;\;\;\text{and} \;\;\; D_{i}^c:=\mathbb{R}_{+}^3\backslash D_{i}.
\end{equation}
\begin{definition}\label{convergence}
Let $H$ be a set of functions such that each $h\in H$ maps from $\mathbb{R} \times \mathbb{R}  \times \mathbb{N}_0$ to $\mathbb{R}$. A sequence $\{F^{i}\}_{i\in \mathbb{N}}$ of distribution functions defined on $\mathbb{R} \times \mathbb{R}  \times \mathbb{N}_0$ is called $F$-convergent with respect to $H$ if each $F^i$ assigns measure to only finitely many values and in addition the following properties hold:
\begin{enumerate}
\item \label{pointwise:conv}$\forall \;  (x,y,l) \in \mathbb{R} \times \mathbb{R}  \times \mathbb{N}_0 , F^{i}(x,y,l) \rightarrow F(x,y,l)$, as $i\rightarrow \infty $.
\item \label{pointwise:unif:approx} Uniformly over $H$:
\begin{enumerate}
\item \label{tail:conv}
$\lim_{i\rightarrow \infty}  \abs{\int_{D^c_{i}} h(x,y,l) \dd F^i (x,y,l)}=0.$
\item \label{int:conv}$\lim_{i\rightarrow \infty}  \abs{\int_{D_{i}} h(x,y,l) \dd F^i (x,y,l)-\int_{D_{i}}  h(x,y,l) \dd F (x,y,l)}=0$.
\end{enumerate}
\end{enumerate}
\end{definition}

Property~\ref{pointwise:conv} is the usual convergence in distribution, while Property~\ref{tail:conv} and \ref{int:conv} ensure that the integral of several functions needed in the following converge. Note that for bounded domains $D\subset D_{\infty}$ and continuous function $h$ it follows that
\begin{equation}
\lim_{i\rightarrow \infty}  \abs{\int_{D} h(x,y,l) \dd F^i (x,y,l)-\int_{D}  h(x,y,l) \dd F (x,y,l)}=0
\end{equation}
by Property~\ref{pointwise:conv} and Helly's theorem. However, the crucial point of  Property~\ref{int:conv} is that the integration domain is becoming larger. In order to understand Property~\ref{tail:conv}, note that for an integrable function $h$ trivially  $\lim_{i\rightarrow \infty}  \abs{\int_{D^c_{i}} h(x,y,l) \dd F (x,y,l)}=0$ holds. Property~\ref{tail:conv} ensures that this convergence holds if $F$ is replaced by $F^i$. It ensures that as $i\rightarrow \infty$ the tail probabilities of the measures implied by $F^i$ are decreasing fast enough for our purpose. 

In the following the functions $F^{i}$ will be the limiting distribution functions of finitary vertex sequences $({\bf w^-}, {\bf w^+}, {\bf c}  )^{i}$ for which Theorem~\ref{fin:weight:res} holds. With increasing integer $i$ the granularity is increasing, that is, there are more weight levels in the sequence and at the same time the approximated range will become larger. The construction of the sequence is done in Section~\ref{seq:construction} and incorporates the set $H$, which will contain functions with unbounded support and therefore Helly's theorem does not allow to conclude convergence of the integral. The assumption that the convergence in Property~\ref{tail:conv} and \ref{int:conv} is uniform will simplify the following analysis. Observe that we do not assume that $\lim_{i\rightarrow \infty}  \abs{\int_{D^c_{i}} h(x,y,l) \dd F (x,y,l)}=0$ uniformly over $H$.
\subsection{Convergence of Some Relevant Functions}\label{sec:conv:funct}
We saw already in Theorem~\ref{fin:weight:res} that the functions $g$ and $f$ defined in (\ref{definition:f}) and (\ref{definition:g}) play a crucial role in determining the final fraction of infected vertices. We show in Proposition~\ref{prop:subseq} that the convergence stated in Definition~\ref{convergence} ensures a certain convergence of these quantities. 

To prove Proposition~\ref{prop:subseq} we need the following lemma.

\begin{lemma}\label{unif:conv}
Let $\{F^{i}\}_{i\in \mathbb{N}}$ be  $F$-convergent with respect to a set of functions $H$ and $\widetilde{H} \subset H$ such that uniformly over $\widetilde{H}$ 
\begin{equation}\label{lemma:uni:ass}
\lim_{i\rightarrow \infty}  \abs{\int_{D^c_{i}} h(x,y,l) \dd F (x,y,l)}=0.
\end{equation}
Then uniformly over $\widetilde{H}$,
\begin{equation}\label{unif:result}
\lim_{i\rightarrow \infty}  \abs{\int_{D_{\infty}} h(x,y,l) \dd F^i (x,y,l)-\int_{D_{\infty}}  h(x,y,l) \dd F (x,y,l)}=0.
\end{equation}
\end{lemma}
\begin{proof}
Let $h \in \widetilde{H} $. By the triangular inequality and since $D_{\infty}=D_{i} \cup D^c_{i},D_{i} \cap D^c_{i}=\emptyset$
\begin{eqnarray}
&&\left\lvert \int_{D_{\infty}}  h(x,y,l) \dd F^{i}(x,y,l) - \int_{D_{\infty}} h(x,y,l) \dd F(x,y,l) \right\rvert \nonumber \\
&\leq & \left\lvert \int_{D_{i}}  h(x,y,l) \dd F^{i}(x,y,l) - \int_{D_{i}} h(x,y,l) \dd F(x,y,l) \right\rvert + \left\lvert \int_{D^c_{i}}  h(x,y,l) \dd F^{i}(x,y,l)  \right\rvert \label{1:Lemma}\\
& +& \left\lvert\int_{D^c_{i}} h(x,y,l) \dd F(x,y,l) \right\rvert \label{2:Lemma},
\end{eqnarray}
where the terms in (\ref{1:Lemma}) converge to zero uniformly over $\widetilde{H}$  since $\widetilde{H} \subset H$ and $\{F^{i}\}_{i\in \mathbb{N}}$ is  $F$-convergent with respect to the set $H$. The term in (\ref{2:Lemma}) converges uniformly by Assumption (\ref{lemma:uni:ass}). This implies (\ref{unif:result}).
\end{proof}
If Condition~\ref{cond:2} of Theorem~\ref{thm:2} holds, that is, there exists $\delta$ such that 
\begin{equation}
\E [ W^+ W^- \P (\poi (z W^-)=C-1)\1_{C\geq 1} ] < \kappa <  1
\end{equation}
for $z \in (\hat{z}-\delta,\hat{z} + \delta)$, let $H$ be defined by
\begin{equation*}
H:=\left( \cup_{z\in [0,2\hat{z}]} h_{1,z}(x,y,l)\right) \cup \left( \cup_{z\in [0,2\hat{z}]} h_{2,z}(x,y,l)\right) \cup \left( \cup_{z\in (\hat{z} - \delta, \hat{z} + \delta )} h_{3,z}(x,y,l)\right)
\end{equation*}
and otherwise we define $H$ by
\begin{equation}
H:=\left( \cup_{z\in [0,2\hat{z}]} h_{1,z}(x,y,l)\right) \cup \left( \cup_{z\in [0,2\hat{z}]} h_{2,z}(x,y,l)\right),
\end{equation}
where the functions $h_{1,z}$ and $h_{2,z}$ are defined by $h_{1,z} (x,y,l):=y \psi_l( x z )$ and $h_{2,z} (x,y,l):= \psi_l( x z )$ respectively.  
Additionally, $h_{3,z}$ is defined by $h_{3,z} (x,y,l):= xy \P (\poi (\hat{z} x)=l-1) $ for $z\in (\hat{z}-\delta,\hat{z}+\delta )$. 

Let $\{F^{i}\}_{i\in \mathbb{N}}$ be  $F$-convergent with respect to the set of functions $H$. To shorten notation we set
\begin{eqnarray*}
f^{i}(z)&:=&f(z;(W^{-},W^{+}, C)^i) \quad \text{and} \quad f (z):=f (z;(W^{-},W^{+}, C))
\end{eqnarray*}
and $g^{i}(y)$ and $g (y)$ accordingly, where $(W^{-},W^{+}, C)^i$ is a random vector with distribution function $F^i$ and $(W^{-},W^{+}, C)$ a random vector with distribution function $F$.\footnote{Since the sequence index $i$ is attached to the entire vector and not to each component, our notation $(W^{-},W^{+}, C)^i$ gives rise to ambiguity in statements about single components of $(W^{-},W^{+}, C)^i$. However, in the following, we shall only consider the entire vector and there is no risk of confusion.}

The following proposition provides the necessary convergence properties of the functions $f$ and $g$ for a $F-$convergent sequence with respect to $H$. The construction of such a sequence is outlined in Section~\ref{seq:construction}.
\begin{proposition}\label{prop:subseq}
Let $\{F^i\}_{i\in \mathbb{N}}$ be $F$-convergent with respect to the set $H$ defined above. Let $\hat{z}$ be the smallest positive zero of $f$ and $\hat{z}^{i}$ the smallest positive zero of $f^{i}(z)$. Then,
\begin{eqnarray}
\liminf_{i\rightarrow \infty}  \hat{z}^{i} & \geq &  \hat{z} \label{lower:bounds:1}\\
\liminf_{i \rightarrow \infty } g^{i} ( \hat{z}^{i} ) & \geq & g(\hat{z} ).\label{lower:bounds:2}
\end{eqnarray}
Furthermore, if Condition~\ref{cond:2} of Theorem~\ref{thm:2} holds, that is, there exists $\delta>0$ such that 
\begin{equation}\label{diff:condi:prop}
\E [ W^+ W^- \P (\poi (z W^-)=C-1)\1_{C\geq 1} ] < \kappa <  1
\end{equation}
for $z\in (\hat{z} - \delta , \hat{z} + \delta )$
, then $f^{'i}(z_i)< 0 $ and
\begin{eqnarray*}
\lim_{i\rightarrow \infty}  \hat{z}^{i} & = &  \hat{z} \quad \text{and} \quad \lim_{i \rightarrow \infty } g^{i} ( \hat{z}^{i} )  = g(\hat{z} ).
\end{eqnarray*}
\end{proposition}
\begin{proof}
We prove the claim if Condition~\ref{cond:2} of Theorem~\ref{thm:2} holds. The lower bounds in (\ref{lower:bounds:1}) and (\ref{lower:bounds:2}) without this assumptions can be shown by similar means, observing that the functions $g^i$ are monotonically increasing in $z$. 

First observe that since $f(z)>0$ for $z\in (0,\hat{z}-\delta )$ we can chose $\delta_0$ as small as we like such that 
\begin{equation*}
f ( \hat{z} - \delta_0 )>0,
\end{equation*}
We apply Lemma~\ref{unif:conv} with the set $\widetilde{H}:= \{ h_{1,z}\}_{z\in [0,2 \hat{z}]}$, where condition (\ref{lemma:uni:ass}) can be easily seen to be satisfied. 
By the definition of $h_{1,z}$ we can conclude that 
\begin{equation*}
\lim_{i \rightarrow \infty } \sup_{z \in [0, 2 \hat{z}]} \abs{f^{i}(z) - f(z) }=\lim_{i \rightarrow \infty } \sup_{z \in [0, 2 \hat{z}]} \abs{(f^{i}(z) +z) - (f(z)+z) }=0. \label{est:2}
\end{equation*}

This observation especially implies that 
\begin{equation*}
2 f ( \hat{z} - \delta_0 ) > f^{i}( \hat{z} - \delta_0 )>0,
\end{equation*}
for $i\geq i_0$.
Furthermore,
\begin{eqnarray}
& & \abs{ \int_{D_{\infty}}xy \P (\poi (z x)=l-1) \dd F^i (x,y,l)} \label{neg:fix:approx}  \\
& =& \abs{ \int_{D_{i}}xy \P (\poi (z x)=l-1) \dd F^i (x,y,l)}+\abs{ \int_{D^c_{i}}xy \P (\poi (z x)=l-1) \dd F^i (x,y,l)} \nonumber \\  \nonumber
\end{eqnarray}
since the integrand is positive. By Assumption (\ref{diff:condi:prop}) there exists a neighborhood $(\hat{z}-\delta, \hat{z} + \delta )$ such that 
\begin{equation}\label{diff:cand:small}
\int_{D_{i}}xy \P (\poi (z x)=l-1) \dd F (x,y,l)\leq \int_{D_{\infty}}xy \P (\poi (z x)=l-1) \dd F (x,y,l) \leq \kappa < 1
\end{equation}
for $z \in  (\hat{z}-\delta, \hat{z} + \delta )$.
Choose now $\delta_1 , \delta_2 >0$ with $\kappa + \delta_1 + \delta_2 < 1$ and $i_1$ such that 
\begin{equation}
\abs{ \int_{D^c_{i}} h(x,y,l) \dd F^i (x,y,l)} < \delta_1
\end{equation}
for each $i\geq i_1$ and $h \in H$. This $i_1$ exists by the Property~\ref{tail:conv} of a $F$-convergent sequence. Further, choose $i_2$ such that 
\begin{equation}
\int_{D_{i}}xy \P (\poi (z_i x)=l-1) \dd F^i (x,y,l) \leq \kappa + \delta_2,
\end{equation}
which exists by Property~\ref{int:conv} of a $F$-convergent sequence and observation (\ref{diff:cand:small}).
This implies that for $i\geq \max \{i_1,i_2 \}$, (\ref{neg:fix:approx}) is bounded by $\kappa + \delta_1 + \delta_2 < 1$. Note that $i_1$ and $i_2$ do not depend on the $\delta_0$ above. Differentiating $f^{i}(z)$, where differentiation under the integral sign is justified by the fact that $F^i$ assigns measure only to finitely many values, implies that 
$f^{' i}(z)< \kappa + \delta_1 + \delta_2 -1 < 0 $
for $ i\geq \max \{i_1,i_2\}$ for $z \in  (\hat{z}-\delta, \hat{z} + \delta )$. 

Now chose $\delta_0<\delta $ such that $-f(\hat{z}-\delta_0)/(\kappa + \delta_1 + \delta_2 - 1)<\delta$, which exists by continuity of $f$ and $f(\hat{z})=0$ and define $\bar{\delta}_0:=-f(\hat{z}-\delta_0)/(\kappa + \delta_1 + \delta_2 - 1)$, then 
\begin{eqnarray}
f^{i}(\hat{z} - \delta_0 + \bar{\delta}_0)& \leq & f(\hat{z} - \delta_0) + \max_{\{ z\in (\hat{z}-\delta,\hat{z}+\delta )\} } \{f^{' i}(z)\}\bar{\delta}_0 \\
& \leq & f(\hat{z} - \delta_0) + (\kappa + \delta_1 + \delta_2 - 1) \bar{\delta}_0 \leq 0
\end{eqnarray}
by the choice of $\bar{\delta_0 }$ and $f$ has a zero in $(\hat{z}-\delta_0, \hat{z}-\delta_0 + \bar{\delta}_0)\subset (\hat{z}-\delta, \hat{z}+\delta )$. Denote by $\hat{z}^{i}$ the first zero of $f^{i}( z )$.
We need to show that $\lim_{i \rightarrow } \hat{z}^{i} =\hat{z}$. Assume for the sake of contradiction that there exists a sub-sequence $\{k_i\}$ with $\hat{z}^{k_i}\in [0,\hat{z} - \delta_0]$. Since $[0,\hat{z} - \delta_0]$ is compact there exists a limit point $\bar{z}$ approached by some subsub-sequence $\{l_{k_i}\}$ such that $\lim_{i\rightarrow \infty} \hat{z}^{l_{k_i}}=\bar{z}\leq \hat{z} - \delta_0$. By continuity, the function $f$ attains its minimum $M$ on $[0,\hat{z}- \delta_0]$. However, by observation (\ref{est:2}), since $f$ is continuous and by
\begin{eqnarray}
& & \abs{f^{l_{k_i}} (\hat{z}^{l_{k_i}})-f(\bar{z})}= \abs{(f^{l_{k_i}} (\hat{z}^{l_{k_i}})-f (\hat{z}^{l_{k_i}}))+(f (\hat{z}^{l_{k_i}}) -f(\bar{z}))} \nonumber \\
& \leq &  \abs{f^{l_{k_i}} (\hat{z}^{l_{k_i}})-f (\hat{z}^{l_{k_i}})} +\abs{f (\hat{z}^{l_{k_i}}) -f(\bar{z})}\nonumber 
\end{eqnarray}
 it follows that 
\begin{equation*}
\lim_{i \rightarrow \infty }\abs{f^{l_{k_i}} (\hat{z}^{l_{k_i}})-f(\bar{z})}=0.
\end{equation*}
On the other hand, however, $\abs{f^{l_{k_i}} (\hat{z}^{l_{k_i}})-f(\bar{z})} =f(\bar{z}) \geq M$, providing the contradiction. Since $\delta_0$ can be chosen arbitrarily small it follows that 
\begin{equation*}
\lim_{i \rightarrow \infty }\hat{z}^{i}=\hat{z}.
\end{equation*}
Further, since $\hat{z}^{i}\in (\hat{z}-\delta ,\hat{z}+\delta)$ for $i$ large, it also follows that $f^{' i}(\hat{z}^{i})<0)$ for $i$ large.
Again, applying Lemma~\ref{unif:conv} with the set $\widetilde{H}:= \{ h_{2,z}\}_{z\in [0,2 \hat{z}]}$ shows that $\lim_{i\rightarrow \infty }\abs{g^{i}(z) - g(z)} =0$ uniformly over $z\in [0,2 \hat{z}]$. Together with $\lim_{i \rightarrow \infty }\hat{z}^{i}=\hat{z}$ and the continuity of $g$, this implies that
\begin{equation}
\lim_{i \rightarrow \infty }\abs{g^{i}(z_i) - g(\hat{z})} \leq \lim_{i \rightarrow \infty } \left(\abs{g^{i}(z_i) - g(z_i)} + \abs{g(z_i) - g(\hat{z})}\right)=0. \label{est:1}
\end{equation}
\end{proof}

\subsection{Constructing a $F$-convergent Sequence}\label{seq:construction}
The following lemma is crucial in our construction of $F$-convergent sequences.
\begin{lemma}\label{convergence:integral}
Let $h$ be a function defined on $[z_{min},z_{max}] \times \mathbb{R}  \times \mathbb{R}\times \mathbb{N}_0$, such that for each $l\in \mathbb{N}_0$, the function $h_l(z,x,y):=h(z,x,y,l)$ is continuous on $[z_{min},z_{max}] \times \mathbb{R}  \times \mathbb{R}$. Let $D\subset \mathbb{R} \times \mathbb{R} \times \mathbb{N}_0$ be a closed, interval set 
 and $\{G_i\}_{i\in \mathbb{N}}$ a sequence of distribution functions on $\mathbb{R} \times \mathbb{R} \times \mathbb{N}_0$ such that for each $(x,y,l)\in D, \lim_{i \rightarrow \infty} F_i (x,y,l) = F (x,y,l)$. Then for every $\varepsilon >0$, there exists $i_0\in \mathbb{N}$ such that for each $i\geq i_0$
\begin{equation}\label{integral:convergence}
\abs{\int_{D} h(z,x,y,l) \dd F^{i}(x,y,l) - \int_{D} h(z,x,y,l) \dd F(x,y,l)} \leq \varepsilon,
\end{equation}
for every $z\in [z_{min},z_{max}]$.
\end{lemma}
\begin{proof}
Choose a step function $\hat{h}(z,x,y,l)$ with $l m$ steps $h_{k,j}, 0 \leq j \leq l, 0 \leq k \leq m$ such that
\begin{eqnarray*}
\hat{h}(z,x,y,l)=h_{k,j}, \text{ for } (z,x,y,l) \in J_k\times I_j
\end{eqnarray*}
where $I_j$ are equally sized interval sets with $D= \cup_j I_j$ and $J_k$ equally sized intervals with $[z_{min},z_{max}]= \cup_j J_k$
such that $\abs{\hat{h}(z,x,y,l)- h(z,x,y,l) }\leq \varepsilon/3$ on $D$ for every $z\in [z_{min},z_{max}]$. This choice is possible due to the fact that both, $D$ and $[z_{min},z_{max}]$ bounded and closed and $h_l(z,x,y)$ continuous for each $l$. Then $\forall i$
\begin{equation*}
\abs{\int_{D} \hat{h}(z,x,y,l) - h(z,x,y,l)  \dd F^{i}(x,y,l) } \leq \varepsilon/3, 
\end{equation*}
and
\begin{equation*}
\abs{\int_{D} \hat{h} (z,x,y,l) - h(z,x,y,l)  \dd F(x,y,l) } \leq \varepsilon/3.
\end{equation*}
Observe that for $z\in J_k$
\begin{eqnarray*}
&&\int_{D} \hat{h} (z,x,y,l) \dd F^{i}(x,y,l) - \int_{D} \hat{h}(z,x,y,l) \dd F(x,y,l) = \sum_j h_{k,j} (\mu^i (I_j) -\mu (I_j) ),
\end{eqnarray*}
where $\mu^i$ and $\mu$ are the measures implied by $F^i$ and $F$ respectively. Since $\lim_{i\rightarrow \infty } \mu^i (I_j)=\mu (I_j) $, by the fact that the $I_j$ are intervals (i.e. the measures $\mu^i (I_j)$ and $\mu (I_j)$ are determined by the values of $F^i$ and $F$ at the endpoints of $I_j$), there exists $i_0$ such that for $i\geq i_0$
\begin{equation*}
\sum_j h_{k,j} (\mu^i (I_j) -\mu (I_j) ) \leq \varepsilon /3.
\end{equation*}
By the triangular inequality it can be easily seen that (\ref{integral:convergence}) holds for $i\geq i_0$.
\end{proof}
Clearly one can chose $k_i$ such that the last lemma holds for a finite set of functions $H$, where each $h\in H$ fulfills the assumptions of the proposition, an observation we shall use in our construction below.

In a first step we construct two $F$-convergent sequences of finitary distribution functions $\{F^{i}_A\}_{i\in \mathbb{N}} $ and $\{F^{i}_B\}_{i\in \mathbb{N}} $ with respect to a given set $H$ of functions. One sequence is such that it generates random graph with asymptotically less infections than in a random graph with limiting distribution $G$ and the second with asymptotically more infections. In a second step we couple a vertex sequence $({\bf w^-}, {\bf w^+}, \bf c  )$ with limiting distribution $F$ to sequences with limiting distribution $F^{i}_A $ and $F^{i}_B$ for each $i\in \mathbb{N}$. This coupling allows for a sandwich type argument in the proof of Theorem~\ref{thm:2}.  For simplicity we pose the following restriction on the limiting distribution of the vertex sequence. The restriction is not mandatory but simplifies the exposition.
\begin{assumption}\label{ass:distrib:fct}
The limiting distribution $F: \mathbb{R} \times \mathbb{R}  \times \mathbb{N}_0^\infty \rightarrow [0,1]$ of the regular vertex sequence $({\bf w^-}, {\bf w^+}, \bf c  )$ is such that for each $l$ the function $F_l(x,y):=F(x,y,l)$ is continuous.
\end{assumption}

Defining a $F$-convergent sequence of finitary distribution functions is rather straightforward under Assumption~\ref{ass:distrib:fct}, but we need a very particular sequence later in the proof of Theorem~\ref{thm:2}.
For this let $H$ be the set of functions defined in Section~\ref{sec:conv:funct} and observe that each $h\in H$ is such that $h_l(x,y):=h(x,y,l)$ is continuous for each $l$. 
For a given $i \in \mathbb{N}$ we determine two distribution functions $F^{i}_A (x,y,l)$ and $F^{i}_B (x,y,l)$ for random vectors $(W^{-},W^{+}, C )_A^{i}$ and $(W^{-},W^{+}, C )_B^{i}$, respectively. The random vector $(W^{-},W^{+}, C )_A^{i}$ is constructed in a way such that 
random graphs approximating $(W^{-},W^{+}, C )_A^{i}$ have more edges and lower threshold values, and in turn allows to approximate the number of infected vertices from above. 
The random vector $(W^{-},W^{+}, C )_B^{i}$ is constructed such that a random graph with limiting distribution $(W^{-},W^{+}, C )_B^{i}$ will have fewer infections than a random graph approximating $(W^{-},W^{+}, C)$ and can be used to approximate the number of infected vertices from below.

For each $i\in \mathbb{N}$ we partition $[x_0 , i]\times [x_0 ,i]$ into $L (i)  L (i)$ equally spaced half open squares $D^i_{k,j}:=[p_k,p_{k+1})\times [p_j,p_{j+1}),1\leq j,k\leq L(i)$, where $L(i)$ is chosen such that $\forall h \in H$
\begin{equation}\label{conv:set:cond}
\abs{\int_{D_{i}} h(x,y,l) \dd \tilde{F}(x,y,l) - \int_{D_{i}} h(x,y,l) \dd F(x,y,l)} \leq 1/i, \text{ for } \tilde{F}\in \{\tilde{F}^i_B, \tilde{F}^i_A \}, 
\end{equation}
where $\tilde{F}_A^{i}$ and $\tilde{F}_B^{i}$ are defined on $D_i$ by
\begin{eqnarray}
\bar{F}_A^{i}(x,y,l)&:=&F(p_{k},p_{j},l)\;\text{if}\;(x,y)\in D^i_{k,j}\\
\bar{F}_B^{i}(x,y,l)&:=&F(p_{k+1},p_{j+1},l)\;\text{if}\;(x,y)\in D^i_{k,j}.
\end{eqnarray}
To see that this choice is possible one can start with sequences $\{\tilde{F}_A^{i,m}\}_{m\in \mathbb{N}}$ and $\{\tilde{F}_B^{i,m}\}_{m\in \mathbb{N}}$, which are defined by partitioning into $mm$ half open cubes for each $l$, and chose $L(i)$ to be the smallest natural number such that (\ref{conv:set:cond}) is fulfilled. To guarantee the existence of such an $L(i)$ one can use Lemma~\ref{convergence:integral} for each of the families $\{h_{i,q}\}, q \in \{1,2,3 \}$) and chose the maximum in an obvious way.

Define the function $\gamma : \mathbb{N} \rightarrow \mathbb{R}_{+}$ by
\begin{equation}
\gamma (i) = \P \left( \{ W^+ > i \} \cup \{ W^- > i \} \cup \{ \infty > C > i \}\right).
\end{equation}
Observe that $\gamma (i)\leq \P (W^+ > i ) +  \P (W^- > i ) + \P (\infty >C > i )$ and $\lim_{i \rightarrow \infty} \gamma (i) =0$.

Now define distribution functions $\{F^{i}_B\}_{i \in \mathbb{N}}$ and $\{F^{i}_A\}_{i \in \mathbb{N}}$ by
\begin{equation}\label{tmp1}
	F^{i}_B(x,y,l):=\left\{
	\begin{array}{ll}
	 0 & \mbox{if } (x<x_0)\vee (y<x_0) \\
	 	\tilde{F}_B^{i}(\min \{x , i\},\min \{y , i\},\min \{ l , i\}) & \mbox{if }   (l<\infty)\wedge ((x\geq x_0)\wedge (y\geq x_0))  \\	
	1  & \mbox{if }  (l=\infty) \wedge ((x\geq x_0)\wedge (y\geq x_0)), \\
	\end{array}
\right.
\end{equation}
where $\wedge$ denotes logical \textit{and} and $\vee $ logical \textit{or}.

To specify $F^i_A$, we determine for each $i$ a value $\bar{w}^+_i$ such that 
\begin{equation}
\bar{w}^+_i \gamma (i) =2 \int_{D_i^c} y \dd F (x,y,l),
\end{equation}
and observe that $\lim_{i\rightarrow \infty }\bar{w}^+_i \gamma (i)=0$. Define
\begin{equation}
	F^{i}_A(x,y,l):=\left\{
	\begin{array}{ll}
	  0 & \mbox{if } (x<x_0)\vee (y<x_0) \\
	 	\tilde{F}_A^{i}(\min \{x , i\},\min \{y , i\},\min \{ l , i\}) & \mbox{if }   (l<\infty)\wedge ((x\geq x_0)\wedge (\bar{w}^+_i \geq y\geq x_0))  \\
	 	\gamma (i)+ \tilde{F}_A^{i}(\min \{x , i\},i,\min \{l , i\}) & \mbox{if }  (l<\infty ) \wedge (x\geq x_0) \wedge (y\geq \bar{w}^+_i ) \\
	1  & \mbox{if }  (l=\infty ) \wedge (x\geq x_0)\wedge (y\geq x_0) .
	\end{array}
\right. 
\end{equation}

\begin{proposition}
The two sequences of distributions $\{F^i_A \}_{i \in \mathbb{N}}$ and $\{F^{i}_B \}_{i \in \mathbb{N}}$ are $F-$convergent with respect to the set $H$.
\end{proposition}
\begin{proof}
Property~\ref{pointwise:conv} of Definition~\ref{convergence} is obvious from the construction of the two sequences $\{F^i_A \}_{i \in \mathbb{N}}$ and $\{F^{i}_B \}_{i \in \mathbb{N}}$. Further by Lemma~\ref{convergence:integral} and the construction we know that the uniform convergence holds over the domain $D_i$, which is exactly Property~\ref{int:conv}. We need to show that Property~\ref{tail:conv} holds.
For $F_B^i$ by construction $\int_{D_i^c} \dd F_B^i (x,y,l)=0$, and therefore for any function $h$ it follows that $\int_{D_i^c} h(x,y,l) \dd F_B^i (x,y,l)=0$. For $F_A^i$ we find 
\begin{eqnarray}
\int_{D_i^c} 1 \dd F_A^i (x,y,l) &=&\gamma (i) \\
\int_{D_i^c} y \dd F_A^i (x,y,l) &=& \tilde{w}^+  \gamma (i) =2 \int_{D_i} y \dd F (x,y,l)\\
\int_{D_i^c} xy \P (\poi (z x)=l-1)  \dd F_A^i (x,y,l) & \leq &\int_{D_i^c} xy  \dd F_A^i (x,y,l) = \tilde{w}_i^+ x_0 \gamma (i)\label{bound:int:diff},
\end{eqnarray}
and all the quantities on the right hand side converge to $0$. Since they are only finitely many, they converge uniformly. Checking the definition of $h_{1,z}, h_{2,z}$ and $h_{3,z}$ shows that the integral with respect to these functions is bounded by one of the above quantities, such that the uniformity holds over $H$.
\end{proof}
\subsection{Coupling to the Original Vertex Sequence}
To allow for a sandwich type argument that squeezes our original sequence with limiting distribution $(W^-,W^+,C)$ between two finitary ones, we will develop a specific coupling. Let as above $({\bf w^-}, {\bf w^+}, \bf c  )$ be a regular vertex sequence with limiting distribution $F: \mathbb{R} \times \mathbb{R}  \times \mathbb{N}_0^\infty \rightarrow [0,1]$ and further $(W^{-},W^{+},C)$ a random vector with distribution function $F$. Recall the definition of $D_i$ in~\eqref{eq:Di}.  For any given $i \in \mathbb{N}$ we construct two finitary regular vertex sequences $({\bf w}^{-}, {\bf w}^{+}, {\bf c}  )_A^{i}$ and $({\bf w}^{-}, {\bf w}^{+}, {\bf c} )_B^{i}$ on the same index set $[n]$ with limiting distributions $F^{i}_A, F^{i}_B: \mathbb{R} \times \mathbb{R}  \times \mathbb{N}_0 \rightarrow [0,1]$ as follows:
\begin{enumerate}
\item Define ${\bf D}^c_{i}:=\{m \in [n] \mid ( w_m^-(n), w_m^+(n), c_m(n)  )\in D^c_{i}\}$. We consider a partition of $[n]\backslash {\bf D}^c_{i}$ into $i \cdot L (i) \cdot L(i)$ parts ${\bf D}_{k,j,l}:=\{m \in [n] \mid ( w_m^-(n), w_m^+(n))\in D_{k,j}, c_m(n)  =l \}$, where $L(i)$ is the number of half open intervals chosen in the definition of $F^{i}_A$ and $F^{i}_B$ (see (\ref{conv:set:cond})).
\item \label{it:2} Construct a vertex sequence $({\bf w}^{-}, {\bf w}^{+}, {\bf c})_{B}^{i}$ on $[n]$ by $(w_{m}^{-},w_{m}^{+}, c_{m}  )_{B}^{i}=(p_k,p_j,l)$ for $m \in {\bf D}_{k,j,l}$ and $(w_{m}^{-}, w_{m}^{+}, c_{m})^{i}_B=(x_0,x_0,\infty )$ for $m \in [n]\backslash {\bf D}^c_{i}$.
\item \label{it:3} Construct a sequence $({\bf w}_{m}^{-}, {\bf w}_{m}^{+}, {\bf c}_{m})_{A}^{i}$ on $[n]$ by $(w_{m}^{-},w_{m}^{+}, c_{m}  )_A^i=(p_{k+1},p_{j+1},l)$ for $m \in {\bf D}_{k,j,l}$ and $(w_{m}^{-}, w_{m}^{+}, c_{m})_A^i=(x_0,\bar{w}_i^+,0)$ for $m \in [n]\backslash {\bf D}^c_{i}$. 
\end{enumerate}
It can easily be seen that the resulting sequences have the required convergence properties, i.e. are regular vertex sequences. 
In \ref{it:2}., the choice of the in- and outweights is irrelevant for vertices with percolation threshold infinity since they cannot spread the contagion process. In \ref{it:3}., for vertices with percolation threshold equal to $0$, the in-weight is irrelevant and was chosen in order to have a bound on the integral in (\ref{bound:int:diff}) and the resulting uniform convergence property for the functions $h_{3,z}$.

\subsection{Proof of Theorem~\ref{thm:2}}\label{proof:thm:2}
Denote by $\mathcal{A}_n ({\bf w^-},{\bf w^+},{\bf c})$ the set of infected vertices for $G ({\bf w^-},{\bf w^+},{\bf c})$, and by $\mathcal{A}_{n} (({\bf w^-},{\bf w^+},{\bf c})_A^{i})$ and $\mathcal{A}_{n} (({\bf w^-},{\bf w^+},{\bf c})_B^{i})$ those for $G(({\bf w^-},{\bf w^+},{\bf c})_A^{i})$ and $G(({\bf w^-},{\bf w^+},{\bf c})_B^{i})$ respectively.
\begin{proof}[Proof of Theorem~\ref{thm:2}]
Again we consider only the slightly more complicated result when $f$ is differentiable and $f^{'} (\hat{z})<0$. The vertex sequence $({\bf w^-},{\bf w^+},{\bf c})_B^i$ has been constructed such that for each vertex $ m \in \cup {\bf D}_{k,j,l}$ the threshold value agrees with its counterpart in $({\bf w^-},{\bf w^+},{\bf c})$ and its in- and outweights are lower. Furthermore, the vertices in $[n] \backslash \cup {\bf D}_{k,j,l}$ are uninfectable in $G(({\bf w^-},{\bf w^+},{\bf c})_B^i)$. These considerations imply that $G(({\bf w^-},{\bf w^+},{\bf c}))$ and $G(({\bf w^-},{\bf w^+},{\bf c})_B^i)$ can be coupled such that 
\begin{equation}\label{lower:bound}
\mathcal{A}_{n} (({\bf w^-},{\bf w^+},{\bf c})_B^{i}) \preceq \mathcal{A}_n (({\bf w^-},{\bf w^+},{\bf c})),
\end{equation}
where $\preceq$ denotes stochastic ordering.

To compare $\mathcal{A}_{n} (({\bf w^-},{\bf w^+},{\bf c})_A^{i})$ with $\mathcal{A}_n (({\bf w^-},{\bf w^+},{\bf c}))$, first note that for each $ m \in \cup {\bf D}_{k,j,l}$ the threshold values are the same for $({\bf w^-},{\bf w^+},{\bf c})$ and $({\bf w^-},{\bf w^+},{\bf c})_A^{i}$ but in and out weight are larger in $({\bf w^-},{\bf w^+},{\bf c})_A^{i}$. Further any $m\in [n] \setminus 
 \left( \cup {\bf D}_{k,j,l} \right)$ is an infected knot in $G(({\bf w^-},{\bf w^+},{\bf c})_A^{i})$. The total out-weight of the vertex set $ [n] \setminus \left( \cup {\bf D}_{k,j,l} \right)$ in $({\bf w^-},{\bf w^+},{\bf c})_A^{i}$ is $2 n \int_{D^c_i} y \dd G(x,y,l)(1+o(1))$ while in $({\bf w^-},{\bf w^+},{\bf c})$ it is $n \int_{D^c_i} y \dd G(x,y,l)(1+o(1))$. This implies that for each vertex $v \in \cup {\bf D}_{k,j,l}$ with in-weight $w^-$, the in-degree $D^-_v$ and $D^-_{v,A}$ of $v$ in ${\bf D}^c_{i}$ can be coupled (similar as in the proof of Theorem~\ref{sec2:degree:seq}) to Poisson random variables $Z_{v}$ and $Z_{v,A}$ (depending on $i$) with parameter $w^- 2 \int_{D^c_i} y \dd G(x,y,l)$ and $w^- \int_{D^c_i} y \dd G(x,y,l)$, respectively, such that
 \begin{eqnarray}
 \P (D^-_v \neq Z_{v})\leq ((w^-)^2+ w^-) o(1)=o(1)\\
  \P (D^-_{v,A} \neq Z_{v,A})\leq ((w^-)^2+ w^-) o(1)=o(1).
 \end{eqnarray}
Since $\P (Z_{v}\geq l) < \P (Z_{v,A}\geq l)$ for all $l \in \mathbb{N}$, and $w_v^-$ and $c_v$ are bounded by $i$, it follows that $\P (D^-_{v}\geq l) < \P (D^-_{v,A}\geq l)$ for $l\leq i$. Since all vertices in ${\bf D}^c_{i}$ are infected in the random graph $G((w^-,w^+,c)_A^{i})$, the probability that the vertex $v$ has egdes to at least $l$ infected vertices in ${\bf D}^c_{i}$ for $l\leq i$ is larger in $G(({\bf w^-},{\bf w^+},{\bf c})_A^{i})$ than in $G(({\bf w^-},{\bf w^+},{\bf c}))$.
These considerations imply that
\begin{equation}\label{coupling:fin:seq}
\mathcal{A}_n (({\bf w^-},{\bf w^+},{\bf c})_B^{i}) \preceq \mathcal{A}_n (({\bf w^-},{\bf w^+},{\bf c})) \preceq \mathcal{A}_n (({\bf w^-},{\bf w^+},{\bf c})_A^{i})).
\end{equation}
Let $(W^{-},W^{+},C)_A^i$ and $(W^{-},W^{+},C)_B^i$ be random vectors distributed according to $F^i_A$ and $F^i_B$ respectively.
Chose two sub-sequences $\{l_{A,i}\}_{i \in \mathbb{N}}\subset \mathbb{N}$  and $\{l_{B,i}\}_{i \in \mathbb{N}}\subset \mathbb{N}$ for the $F$-convergent sequences $F^i_A$ and $F^i_B$ as provided by Proposition~\ref{prop:subseq} such that $\lim_{i \rightarrow \infty }\hat{z}_A^{l_{A,i}}= \lim_{i \rightarrow \infty }\hat{z}_B^{l_{B,i}} =\hat{z}$, where $\hat{z}_A^{l_{A,i}}$ and $\hat{z}_B^{l_{B,i}}$ are the smallest zeros of $f(z,(W^{-},W^{+},C)_A^{l_{A,i}})$ and $f(z,(W^{-},W^{+},C)_B^{{l_{B,i}}})$ respectively.. Additionally $\lim_{i \rightarrow \infty }g(\hat{z}_A^{l_{A,i}},(W^{-},W^{+},C)_A^{l_{A,i}})= \lim_{i \rightarrow \infty }g (\hat{z}_B^{l_{B,i}},(W^{-},W^{+},C)_B^{l_{B,i}}) =g(\hat{z} )$ holds. By Theorem~\ref{fin:weight:res} together with (\ref{coupling:fin:seq}), it follows that for all $\varepsilon > 0$ and all $i \in \mathbb{N}$
\begin{equation*}
\lim_{n \rightarrow \infty}\P \left(g (\hat{z}_B^{l_{B,i}},(W^{-},W^{+},C)_B^{l_{B,i}}) -\varepsilon  \leq n^{-1}\abs{\mathcal{A}_n ({\bf w^-},{\bf w^+},{\bf c})} \leq  g(\hat{z}_A^{l_{A,i}},(W^{-},W^{+},C)_A^{i})+\varepsilon \right)=1,
\end{equation*}
proving Theorem~\ref{thm:2}.
\end{proof}

\section{Applications} \label{sec:apps}
\subsection{Quantifying Systemic Risk}
\label{sec:qsr}

We begin with the proof of Theorem~\ref{thm:epsilon:grow}.

\begin{proof}[Proof of Theorem~\ref{thm:epsilon:grow}]
Since $\psi_0 (x)=1$ for $x>0$ and especially $\psi_0 (x) > \psi_r (x)$ for $r\in \mathbb{N}^\infty$ and further $W^+$ strictly positive, it follows that 
\begin{equation}\label{ineq:chain}
\E [W^+ \psi_{C_M} (z W^-)] > \E [W^+ \psi_C (z W^-)]\geq z ,
\end{equation}
such that $\E [W^+ \psi_{C_M} (z W^-)]>0$ for $z \in (0,z_0 ]$. To see that the left inequality in (\ref{ineq:chain}) is really strict, choose $\bar{w}$ with $\P ((W^- \leq \bar{w}) \wedge  (M=0)) \geq \P (C_M = 0)/2$, then it follows that 
\begin{equation*}
\E [W^+ \psi_{C_M} (z W^-)] - \E [W^+ \psi_C (z W^-)] \geq x_0 (1/2)\P (C_M = 0) (1-\psi_{1} (z \bar{w})) >0.
\end{equation*}
By continuity of $\E [W^+ \psi_C (z W^-)]$ (compare with the proof of Lemma~\ref {f:continuous}) it is easy to see that $\E [W^+ \psi_C (z W^-)]>0$ for $z \in (0,z_0 + \delta ]$ for some $\delta > 0$. Furthermore, since $\P (C_M = 0)>0$ and $W^+\geq x_0$, it follows that $\E [W^+ \psi_{C_M} (0 W^-)]>0$. Let $\hat{z}$ be the smallest positive solution of 
\begin{equation*}
f(z;(W^{-},W^{+},C_M))=0.
\end{equation*}
The above considerations imply that $\hat{z} > z_0$. Similarly $\E[\psi_{C_M} (W^- \hat{z} )] > \E[\psi_{C} (W^- z_0 )]$. By Theorem~\ref{thm:2} for any $\varepsilon >0$
\begin{equation*}
\lim_{n \rightarrow \infty} \P ( n^{-1} \abs{\mathcal{A}_n } <  \E[\psi_{C_M} (W^- \hat{z} )] - \varepsilon )=0.
\end{equation*}
Choose $\varepsilon= \E[\psi_{C_M} (W^- \hat{z} )] - \E[\psi_{C} (W^- z_0 )]>0$ and (\ref{small:fraction:bound}) follows.
\end{proof}
A sufficient condition for the assumptions of Theorem~\ref{thm:epsilon:grow} to hold is that $f$ is right differentiable in $0$ with positive derivative. To see that the bound in (\ref{small:fraction:bound}) can in general not be improved consider a vertex sequence with limiting distribution $(W^-,W^+,C)$ fulfilling (\ref{cond:small:sets}) with $f(z;(W^{-},W^{+},C))$ continuously differentiable and such that $f' (\hat{z};(W^{-},W^{+},C))<0$ where $\hat{z}>0$ is the smallest strictly positive solution of 
\begin{equation*}
f(z;(W^{-},W^{+},C))=0.
\end{equation*}
Now infect ex post all vertices i.i.d. with probability $p >0$. The resulting vertex sequence is close to $(W^{-},W^{+},C_{M_p})$ with $C_{M_p}:=C M_p$, where $M_p$ is a Bernoulli random variable independent of all others with success probability $1-p$. Conditioning on $M_p$ shows that
\begin{equation*}
f(z;(W^{-},W^{+},C_{M_p}))= (1-p) \E [W^+ \psi_C (zW^-)] + p \E [W^+] - z.
\end{equation*}
Choose $\delta_0 > 0$ such that $f' (z;(W^{-},W^{+},C))<0$ for all $z\in (\hat{z}-\delta_0,\hat{z}+\delta_0)$. Since we have that   $\frac{\partial}{\partial z} \E [W^+ \psi_C (zW^-)] \geq 0$, it follows that
\begin{equation}\label{diff:smaller:0}
f' (z;(W^{-},W^{+},C_{M_p})) \leq f' (z;(W^{-},W^{+},C))<0,
\end{equation}
for $z\in (\hat{z}-\delta_0,\hat{z}+\delta_0)$.
Let $\hat{z}_p\geq \hat{z}$ be the first positive solution of $f(z;(W^{-},W^{+},C_{M_p}))=0$. Then one easily observes that $\lim_{p \rightarrow 0} \hat{z}_p =\hat{z} $ and $\lim_{p \rightarrow 0} \E [ W^+ \psi_{C_{M_p}} (\hat{z}_p W^-)]  =\E [ W^+ \psi_C (\hat{z} W^-)] $. This implies that there exists $p_0>0$ such that $\hat{z}_p\in (\hat{z}-\delta_0,\hat{z}+\delta_0)$ for $p<p_0$ and therefore $f' (\hat{z}_p;(W^{-},W^{+},C_{M_p}))<0$ by (\ref{diff:smaller:0}). By Theorem~\ref{thm:2}, for any $\delta > 0$ there exists $p$ such that with high probability
\begin{equation*}
n^{-1}\abs{\mathcal{A}_{n} } \leq  \E[\psi_C (W^- \hat{z} )] + \delta ,
\end{equation*}
which shows that the bound in (\ref{small:fraction:bound}) is best possible.
The following corollary shows that in a network satisfying the requirements of Theorem~\ref{thm:epsilon:grow}  a sublinear set of initially infected vertics is sufficient for the infection to spread to a linear set.
\begin{corollary}\label{corr:threshold:compare}
Let $({\bf w^-}, {\bf w^+}, \bf c  )$ be a regular vertex sequence and $(W^{-},W^{+},C)$ a random vector with distribution as the limiting distribution of $({\bf w^-}, {\bf w^+}, \bf c  )$. Assume that $(W^-,W^+,C)$ satisfies \eqref{cond:small:sets} for $z<z_0$. Then, there exists a sequence $\varepsilon (n)$ with $\lim_{n \rightarrow \infty }\varepsilon (n) =0$ such that if we infect each vertex $i \in [n]$ independently with probability $\varepsilon (n)$, then with high probability
\begin{equation*}
n^{-1} \abs{\mathcal{A}_n }\geq  \E[\psi_C (W^- z_0 )] > 0,
\end{equation*}
where $\mathcal{A}_n$ is the final set of infected vertices.
\end{corollary}
\begin{proof}
Let $\varepsilon_i:=1/i$ and as before $M_{\varepsilon_i}$ a Bernoulli random variable independent of all other variables with success probability $1-\varepsilon_i$. By Theorem~\ref{thm:epsilon:grow}, we can define $n_i$ such that 
\begin{equation*}
\P \left( n^{-1} \abs{\mathcal{A}_n } < \E[\psi_C (W^- z_0 )] \right) \leq 1/i 
\end{equation*}
for $n\geq n_i$ and $n_i> n_{i-1}$ (to ensure that the sequence $\{n_i\}_{i \in \mathbb{N}}$ is strictly increasing) in the random graph parametrized by $(W^-,W^+,C_{M_{\varepsilon_i}})$. Define $\varepsilon (n)=\varepsilon_i$ for $n_{i} \leq n < n_{i+1}$.
\end{proof}

We proceed with the proof of Theorem~\ref{prop:resilient}.
\begin{proof}[Proof of Theorem~\ref{prop:resilient}]
Let $\varepsilon >0$, choose $\delta$ such that $g (z; (W^-,W^+,C)) \leq \varepsilon/2$ for $z \in [0,\delta]$. This choice is possible since $g$ is continuous by Lemma~\ref{f:continuous} and $g (0; (W^-,W^+,C))=0$.
In order to prove the claim, we first show that there exists $j_0$ such that $f (z,(W^-,W^+,C_{M^{(j)}}))$ has a fixpoint $\hat{z}_{M^{(j)}}\leq \delta$ for $j \geq j_0$ with
\begin{equation}
\E [ W^+ W^- \P (\poi (\hat{z}_{M^{(j)}} W^-)=C_{M^{(j)}}-1)\1_{\{C_{M^{(j)}}\geq 1 \}} ] <  1.
\end{equation}
Let $\delta_2  := \min \{ \delta , z_0 \} $.
By Assumption~\ref{cond:small:f}, we have that $f (\delta_2,(W^-,W^+,C)) < 0$, since $\delta_2\leq z_0$. Further,
\begin{eqnarray}
& & f (\delta_2,(W^-,W^+,C_{M^{(j)}}))\nonumber \\
&=& f (\delta_2,(W^-,W^+,C)) + \E [W^+ ( \psi_{C_{M^{(j)}}} (\delta_2 W^-) - \psi_C (\delta_2 W^-)) ]\nonumber \\
& \leq & f (\delta_2,(W^-,W^+,C)) + \E [W^+ \1_{\{M^{(j)}=0\}} ] \label{choose:i:resilient}
\end{eqnarray}
Since $\lim_{i \rightarrow \infty}\P(M^{(j)}=0)=0$ it follows that $\lim_{i \rightarrow  \infty }\E[W^+ \1_{\{M^{(j)}=0\}}]=0$ by $\E[W^+]<\infty$. This observation allows us to choose $j_0$ such that $\E [W^+ \1_{\{M^{(j)}=0\}} ] < - f (\delta_2,(W^-,W^+,C))$ and the right hand side of (\ref{choose:i:resilient}) is strictly smaller than $0$. This ensures that $\hat{z}_{M^{(j)}}\leq \delta_2 $ for $j\geq j_0$, where $\hat{z}_{M^{(j)}}$ is the first positive zero of $f (z,(W^-,W^+,C_{M^{(j)}}))$, which exists by \ref{f:continuous}. Furthermore, by the choice of $\delta_2$
\begin{equation*}
\E [ W^+ W^- \P (\poi (\hat{z}_{M^{(j)}} W^-)=C-1)\1_{\{C\geq 1 \}} ] <  1. 
\end{equation*}
As a consequence, since $P ( ( C_{M^{(j)}} =C) \cup (C_{M^{(j)}}=0))=1$,
\begin{equation*}
\E [ W^+ W^- \P (\poi (\hat{z}_{M^{(j)}} W^-)=C_{M_i}-1)\1_{\{C_{M^{(j)}}\geq 1 \}} ] <  1. 
\end{equation*}
The monotony of $g$ and $\hat{z}_{M^{(j)}}<\delta $ implies $g (\hat{z}_{M^{(j)}}; (W^-,W^+,C))<g (\delta; (W^-,W^+,C))\leq \varepsilon / 2$. Further
\begin{equation}
g (\hat{z}_{M^{(j)}}; (W^-,W^+,C_{M_i})) \leq g (\hat{z}_{M^{(j)}}; (W^-,W^+,C)) + \P (M^{(j)}=0),
\end{equation}
and choosing $j_{\varepsilon}\geq j_0$ such that $\P (M^{(j)}=0) < \varepsilon/2$ for all $j\geq j_{\varepsilon}$ the claim follows again by Theorem~\ref{thm:2}.
\end{proof}


\subsection{Examples}\label{sec:Ex}
In this section we  give some example distributions of random vectors $(W^-,W^+,C)$. The first example shows that in the absence of a second moment of $(W^-,W^+)$, condition (\ref{cond:small:sets}) can be satisfied even when not a single vertex with threshold value $1$ exists. 

\begin{example}\label{non:res:no:thres1}
Assume that $W=W^-=W^+$ and $\P (C=r)=1$ for some $r \in \mathbb{N},r\geq 2$. Furthermore, let $W$ be power-law distributed with exponent $\beta \in (2,3)$, that is, the density function $h(w)$ is given by
\begin{equation*}
	h(w):=\left\{
	\begin{array}{ll}
	 Cw^{-\beta}, & \mbox{if } w\geq 1 \\
	0 , & \mbox{else } \\
	\end{array},
\right.
\end{equation*}
where $C$ is a normalizing constant. Since $\psi_r (zW) > e^{-zW} \frac{(zW)^r}{r \!! }$, it follows that
\begin{eqnarray*}
\E [W^+ \psi_C (zW^- )] & \geq & \frac{z^r}{r\!!}C\int_1^{\infty} w^{(r+1)-\beta }e^{-zw} \dd w \nonumber \\
& \geq & z^r C_1 \int_1^{1/z} w^{(r+1)-\beta }e^{-zw} \dd w \nonumber\\
& \geq & z^r C_1 \int_1^{1/z} w^{(r+1)-\beta } \dd w  \geq  C_1(z^{\beta - 2}-z^r),\label{est:pl}  \nonumber
\end{eqnarray*}
where it was used that $e^{-zw}\geq e^{-1}$ for $w\leq 1/z$ and $C_1$ is a constant changing from line to line but independent of $z$. Now choose $z_0$ such that $C_1(z^{\beta - 2}-z^r)>z$ for $z \in (0,z_0)$.
\end{example} 
The last example can be easily generalized to the situation where $\P (C=r)>0$ for some $r\geq 2,r\in \mathbb{N}$ and $W | (C=r)$ is power-law distributed with parameter $\beta \in (2,3)$. In contrast, note that if we choose the exponent of the power-law to be $> 3$, then a simple calculation shows that the assumptions of Theorem~\ref{prop:resilient} are satisfied, and thus the network will be always resilient (see also~\cite{Cont2013b}).

In our next example we choose the parameters such that $\E \left[ W^+ W^- \P \left( \poi (z W^- ) =  C-1 \right)\1_{C>0} \right] \leq \kappa <1$ in a neighborhood of $\hat{z}$, the fixpoint of $f$ and Theorem~\ref{thm:2} allows to determine the exact fraction of infected vertices at the end of the process.
\begin{example}\label{example:1}
Let $(\Omega,\mathcal{F},\mathbb{P})$ be a probability space with positive random variable $W\in \mathbb{R}$ and random vector $(W^-,W^+,C)\in \mathbb{R}_+ \times \mathbb{R}_+ \times \mathbb{N}$ defined. Let further $W^+\in \mathcal{L}^p, p>1$ be independent of $\{W^-,C \}$ given $W$ ($W^+ \independent_{\sigma (W)} ( W^-,C ))$ and let $C \vert W^-$ be uniformly distributed on $\{0,1,2,\dots, \lceil W^- \rceil \}$. Further, we assume that there exists a $p>1$ such that $\E [(W^+)^p (W^-)^{p-1}]<\infty$.

Let $\hat{z}$ be the first positive zero of $f(z,(W^-,W^+,C)) $. We show that $f'(\hat{z},(W^-,W^+,C)) < 0$. For this we differentiate $f(z,(W^-,W^+,C) $ below the integral sign and obtain that
\begin{equation}\label{diff:candidate}
\E \left[\frac{\partial W^+ \psi_C (z W^- )}{\partial z} \right]= 
\E \left[ W^+ W^- \P \left( \poi (z W^- ) =  C-1 \right)\1_{C>0} \right].
\end{equation}
Conditioning on $W$ and using that $W^+ \independent_{\sigma (W)} ( W^-,C ))$ yields
\begin{eqnarray}
&& \E \left[ (W^+ W^- \P \left( \poi (z W^- ) =  C-1 \right)\1_{C>0})^p \right]\nonumber \\
&& \E[ \E [ (W^+)^p \vert W] \E [ (W^- \P \left( \poi (z W^- ) =  C-1 \right)\1_{\{C>0\}})^p \vert W ] ]\nonumber \\
& = & \E[ \E [ (W^+)^p \vert W] \E [ \E[ (W^- \P \left( \poi (z W^- ) =  C-1 \right)\1_{\{C>0\}})^p \vert W,W^-] \vert W ] ].\label{example1:eq:1}
\end{eqnarray}
Since the distribution of $C \vert W, W^-$ is uniform on $\{0,1,2,3,\dots, \lceil W^- \rceil \}$ using the disintegration theorem (\cite[Thm. 6.4.]{Kallenberg2001}), it follows that
\begin{eqnarray}
\E[ (W^-)^p \P \left( (\poi (z W^- ) =  C-1 \right)\1_{C>0})^p \vert W,W^-]& =&  (W^-)^p  \E[(\P \left( \poi (z W^- ) =  C-1 \right)\1_{C>0})^p \vert W,W^-]\nonumber \\
&=& (W^-)^p  \sum_{c=1}^{\lceil W^- \rceil} \frac{1}{\lceil W^- \rceil +1}(\P \left( \poi (z W^- ) =  c-1 \right))^p \nonumber \\
&\leq & \frac{(W^-)^p}{\lceil W^- \rceil +1} (\P \left( \poi (z W^- ) \leq \lceil W^- \rceil -1 \right))^p \nonumber \\
&\leq  &   (W^-)^{p-1}\nonumber 
\end{eqnarray}
These considerations imply that 
\begin{equation}\label{diff:bound}
\E[ (W^+ W^- \P \left( \poi (z W^- ) =  C-1 \right)\1_{C>0})^p ]\leq \E [(W^+)^p (W^-)^{p-1}]<\infty
\end{equation}
for all $z$, which justifies differentiation under the integral sign by the Vitali convergence theorem and shows that
\begin{equation}\label{diff:f}
f^{'}(z,(W^+,W^-,C))=\E[  W^+ ] \E[\frac{W^-}{\lceil W^- + 1 \rceil} \P \left( \poi (z W^- ) \leq \lceil W^- \rceil -1 \right) ]] -1.
\end{equation}
This observation implies that the derivative is strictly decreasing in $z$. Furthermore, since $\P(C=0)>0$, it follows that $f(0,(W^+,W^-,C))>0$ and by $f(\hat{z})=0$ that $f^{'}(z,(W^+,W^-,C))<0$ in a neighborhood of $\hat{z}$. Equality (\ref{diff:candidate}), which is now justified, implies that 
\begin{equation*}
\E \left[ W^+ W^- \P \left( \poi (z W^- ) =  C-1 \right)\1_{C>0} \right]<1
\end{equation*}
in a neighborhood of $z$.
The example can be adjusted slightly by choosing $C \vert W, W^-$ uniform on $\{1,2,3,\dots, \lceil W^- \rceil \}$, then $\P (C=0)=0$ and analog observations as above show that for $\E [W^+]<1$ the derivative $f^{'}(z,(W^+,W^-,C))$ is negative for all $z$ and the network is resilient according to Theorem~\ref{prop:resilient}. Furthermore, one easily observes that the assumption $\E [(W^+)^p (W^-)^{p-1}]<\infty$ can be dropped if $C \vert W^-$ is uniformly distributed on $\{0,1,2,\dots, \lceil (W^-)^p \rceil \}$ for some $p>1$.
\end{example}
A particularly interesting case in the previous example is given by choosing $W$ power law with some parameter $\alpha$ and $W^-\vert W$ and $W^+\vert W$ both Poisson distributed with parameter $W$. Then both $W^-$ and $W^+$ are power law distributed with parameter $\alpha$ (see e.g.~\cite{Hofstad2014} for a related situation).

In our final example (\ref{cond:small:sets}) is satisfied and therefore any small infection spreads to a positive fraction of the random graph. In contrast to Example~\ref{non:res:no:thres1}, the infection is spread mainly by vertices with threshold function $1$ and the distribution of $(W^-,W^+)$ can have all moments.
\begin{example}
We adjust Example~\ref{example:1} slightly such that $\P (C=0)=0$ and $\P (C=1 \vert W^{-}) \geq  (1+ \delta )/W^{-} $ for some $\delta >0$ (this can be done for example by choosing $\P (W^{-}\geq 1+ \delta)=1$ and the conditional distribution of $C \vert (W,W^-)$ such that $\P ( C =1\vert (W,W^-) )= (1+ \delta )/W^{-}$ and $\P ( C =j\vert (W,W^-) )= ( 1-(1+ \delta )/W^{-})/(\lceil W^{-} \rceil - 1 ))$ for $j \in \{2,3,\dots, \lceil W^- \rceil \}$. Furthermore, assume that $\E[W^+]>1$. To see that Condition \ref{cond:small:sets} holds, observe that
\begin{eqnarray}
\E [\E [W^+ \P (\poi (\varepsilon W^-)\geq C  )\vert (W,W^-)]]& \geq & \E[ \E [W^+ \vert (W,W^-)]  \E [ \P (\poi (\varepsilon W^-)\geq 1 ) \frac{1+ \delta }{W^{-}} ] \nonumber \\
&=& \E [ \E [W^+ \vert W] e^{- \varepsilon W^{-}}\varepsilon W^{-}\frac{1+ \delta }{W^{-}}]\nonumber \\
&\geq & \varepsilon \E [ \E [W^+ \vert W] e^{- \varepsilon \bar{W}} (1+ \delta )\1_{\{W^{-} \leq \bar{W} \}}].
\end{eqnarray}
Choose $\bar{W}$ such that $\E[\E [W^+ \vert W]  \1_{\{ W^{-} \leq \bar{W} \}}]\geq (1-\delta_1) \E[W^{+}]$ for some $\delta_1$ with $(1-\delta_1)(1+ \delta )>1$. Then there exists $\varepsilon_0$ such that for  $\varepsilon \leq \varepsilon_0$ this quantity is larger than $\varepsilon (1+\delta_2)$ for some $\delta_2>0$ by $\E [ W^{+}]>1$. Condition \ref{cond:small:sets} follows.
\end{example}

\section{Extensions}\label{extension:relevance}

In order to define measures of resilience that incorporate the relevance of each vertex we shall extended the setting of a regular vertex sequences and state an extension of the Theorem~\ref{thm:2} and Theorem~\ref{prop:resilient} for this setting. In particular, let
\begin{equation*}
\mathcal{R}_n=\sum_{i \in \mathcal{A}_n} r_i,
\end{equation*}
the loss in relevance due to the infection. As mentioned in Section~\ref{sec:resapp} we will be interested in $\lim_{n\rightarrow \infty }  (\mathcal{R}_n \big/ \sum_{i \in [n]} r_i  )$, the fractional loss in relevance. Theorem~\ref{thm:2} will then arise as a special case by choosing $r_i=1$ for $i\in [n]$.
Since the proof can be done without any further mathematical difficulties but notation complicates, we shall only outline the changes necessary in Sections~\ref{sec:finitary}, \ref{sec:approximation} and \ref{sec:qsr}.

In addition to the weights and threshold values, we assume now that each vertex $i\in [n]$ is associated with a relevance value $r_i\in \mathbb{R}_+$. As in the case of weights and threshold values, we assume that we are given a relevance sequence ${\bf r} (n)=(r_1 (n),\dots ,r_n (n)) \in (\mathbb{N}_0)^n$. Again, we need some kind of regularity
\begin{definition}[\bf Extended Regular Vertex Sequence]\label{ext:reg:vertex:seq}
We call $({\bf w^-}, {\bf w^+},{\bf r} ,\bf c  )$ a \emph{regular extended vertex sequence} if $({\bf w^-}, {\bf w^+},\bf c  )$ is a regular vertex sequence and there exists a distribution function $F: \mathbb{R}^2 \times \mathbb{R}_+ \times \mathbb{N}_0^\infty \rightarrow [0,1]$ such that for all points $(v,x,y,l) \in \mathbb{R}^2 \times \mathbb{R}_+ \times \mathbb{N}_0^\infty$ for which $F(v,x,y,l)$ is continuous in $(v,x,y)$ we have $\lim_{n\rightarrow \infty} F_n(v,x,y,l)=F(v,x,y,l)$, where $F_n(v,x,y,l)$ is the empirical distribution function defined by
\begin{equation*}
F_n(v,x,y,l)=n^{-1}\sum_{i \in [n]} \1 \{w^-_i(n) \leq v, w^+_i(n) \leq x, r_i(n)  \leq y , c_i(n)  \leq l \}, \;\; \forall (v,x,y,l) \in \mathbb{R} \times \mathbb{R}  \times \mathbb{N}_0^\infty.
\end{equation*} 
\end{definition}
Let in the following $({\bf w^-}, {\bf w^+},{\bf r} ,\bf c  )$ be a regular extended vertex sequence and $(W^-,W^+,R,C)$ a random variable distributed according to its limiting distribution. Further, we assume that $\E [R]< \infty$. If we can determine $\lim_{n\rightarrow \infty } n^{-1} \mathcal{R}_n $, we can also determine the fractional loss in relevance (\ref{rel:loss:rel}), which is then $\lim_{n\rightarrow \infty } (\E[R] n)^{-1} \mathcal{R}_n $.
A careful analysis of the proof of Theorem~\ref{thm:2} shows that in addition to the sets $I_{j,k;m}^{l}(t)$ for $1\leq j\leq l_1, 1\leq k\leq l_2,0 \leq l<m \leq c_{\max}$ we can also define $I_{j,k;m}^{m}(t)$ for $1\leq j\leq l_1, 1\leq k\leq l_2, m \leq c_{\max}$, the set of infected vertices with weight levels $\tilde{w}^-_j$ and $\tilde{w}^+_k$ and percolation threshold $m$. That way we also keep track of the types of defaulted vertices and not only their average weight. Let further $c_{j,k;m}^{m}(t)$ be their size. Their expected evolution is then given by
\begin{equation*}
\E[c^{m}_{j,k;m}(t)-c^{m}_{j,k;m}(t-1)\vert  h(t-1)] = n^{-1}\left(\1_{\{l\neq 0\}}c^{m-1}_{j,k;m}(t-1) \right) \left(\frac{\tilde{w}_j^- w (t-1)}{u(t-1)}\right),
\end{equation*}
where $h(t)$ is now the state of the extended system:
\begin{equation*}
h (t)=\left(u(t),w (t),\{c_{j,k;m}^{l}(t)\}_{\substack{j \in [l_1], k\in [l_2]\\0 \leq l \leq m \leq c_{\max}} }\right).
\end{equation*}
The functions $\gamma^{m}_{j,k;m}, 1\leq j\leq l_1, 1\leq k\leq l_2,m \leq c_{\max}$  approximating the additional quantities solve the differential equation
\begin{equation*}
\frac{\dd \gamma^{m}_{j,k;m}(\tau)}{\dd \tau} = \left(\1_{\{l\neq 0\}}\gamma^{m-1}_{j,k;m}(\tau)\right) \left(\frac{\tilde{w}_j^- \mu (\tau)}{\nu (\tau)}\right).
\end{equation*}
As 
\begin{equation*}
u(t)  = \sum_{\substack{j \in [l_1], k\in [l_2]\\0 \leq l \leq m \leq c_{\max}} }  c^{m}_{j,k;m}(t) - t,
\end{equation*}
one easily observes from (\ref{solution:nu}) and (\ref{solution:mu}) that 
\begin{equation*}
\gamma^{0}_{j,k;m}(\tau) = \gamma^{0}_{j,k;m}(0) \P \left[ \text{Poi} \left( \tilde{w}_j^- z(\tau) \right) \geq m \right].
\end{equation*}
The remainder of the proof of Theorem~\ref{thm:2} can then be applied without major changes. Assuming that
\begin{equation}\label{fixpoint:ext:setting}
\E [ W^+ W^- \P (\poi (z W^-)=C-1)\1_{C\geq 1} ] < \kappa <  1\text{ for }z \in (\hat{z}-\delta, \hat{z}+\delta )
\end{equation}
and some $\delta >0$, the convergence properties inherent in Definition~\ref{ext:reg:vertex:seq} allow to conclude that
\begin{eqnarray}
&&n^{-1} \mathcal{R}_n=n^{-1} \sum_{i \in \mathcal{A}_n} r_i \nonumber \\
&=& (1+ o_p(1)) \sum_{\substack{j \in [l_1], k\in [l_2]\\0 \leq l \leq m \leq c_{\max}} } \gamma^{0}_{j,k;m}(0) \P \left[ \text{Poi} \left( \tilde{w}_j^- z(\hat{\tau }) \right) \geq m \right] \E [R | (W^-,W^+,C)]\label{ext:R}\\
&=& (1+ o_p(1)) \E[R \psi_C  ( W^- \hat{z}) ]\nonumber 
\end{eqnarray}
where $\gamma^{0}_{j,k;m}(0)=\P(W^-=\tilde{w}_j^-,W^+=\tilde{w}_k^+,C=m)$. It follows that 
\begin{equation*}
n^{-1}\mathcal{R}_n \xrightarrow{p} \E [R \psi_C( W^- \hat{z} ) ], \text{ as } n\rightarrow \infty .
\end{equation*}
In case (\ref{fixpoint:ext:setting}) does not hold, Equation (\ref{ext:R}) becomes an inequality $(\leq )$  and it follows that for all $\epsilon>0$ with high probability: 
\begin{equation*}
n^{-1} \mathcal{R}_n \geq  \E [ R \psi_C( W^- \hat{z} ) ]-\epsilon.
\end{equation*}
This shows a counterpart of Theorem~\ref{fin:weight:res} for $\mathcal{R}_n$ instead of $\abs{\mathcal{A}_n}$. Note that in the steps outlined above it was not used that $R$ attains only finitely many values as it is the case for $W^-$ and $W^+$ in Sections~\ref{sec:finitary}. This is due to the fact that the relevance factor $R$ has no influence on the infection process. For this reason a generalization to general weight sequence can be done exactly as before with a sequence $\{(W^-,W^+,R,C)^i\}_{i \in \mathbb{N}}$, where the third entry can have the original marginal distribution of $R$ for all $i \in \mathbb{N}$.

\begin{theorem}\label{main:thm:ext}
Let $({\bf w^-}, {\bf w^+}, {\bf r}, \bf c  )$ be an extended regular vertex sequence with limiting distribution $F: \mathbb{R}^2 \times \mathbb{R}_+ \times \mathbb{N}^\infty_0 \rightarrow [0,1]$ such that $F(v,x,y,l)$ is continuous in $(v,x,y)$. Further, let $\P (C=0)>0$.
Denote by $\hat{z}$ the smallest positive solution of 
\begin{equation*}
f(z;(W^{-},W^{+},C))\stackrel{!}{=}0,
\end{equation*}
with $f$ as defined in (\ref{definition:f}). Then the following holds:
\begin{enumerate}
\item For all $\epsilon>0$ with high probability: 
\begin{equation*}
n^{-1} \mathcal{R}_n \geq  \E [ R \psi_C( W^- \hat{z} ) ]-\epsilon.
\end{equation*} 
\item \label{cond:22} If there exists $\delta$ such that $\E [ W^+ W^- \P (\poi (z W^-)=C-1)\1_{C\geq 1} ] < \kappa <  1$ for $z \in (\hat{z}-\delta, \hat{z}+\delta )$, then
\begin{equation*}
n^{-1}\mathcal{R}_n \xrightarrow{p}\E [ R \psi_C( W^- \hat{z} ) ], \text{ as } n\rightarrow \infty .
\end{equation*}
\end{enumerate}
\end{theorem}

Let as before $({\bf w^-}, {\bf w^+}, \bf r  , \bf c  )$ be an extended regular vertex sequence and let $F$ be its limiting distribution, and further $(W^-,W^+,R,C)$ a random variable with distribution $F$. We now assume that $\P (C=0)=0$.

Similarly we state extensions of Theorem~\ref{thm:epsilon:grow} and Theorem~\ref{prop:resilient} whose relevance for a financial context is apparent in \cite{Detering2015}. 

\begin{theorem}\label{thm:epsilon:grow:R}
Assume that $(W^-,W^+,C)$ is such that there exists $z_0>0$ such that for any $0< z < z_0$
\begin{equation}\label{cond:small:sets1}
\E [W^+ \psi_C (z W^-)] > z.
\end{equation}
Let $M$ be such that $\P (M = 0)>0$. Let $\mathcal{R}_n$ be the loss in relevance in $G_n({\bf w}^-, {\bf w}^+, \overline{\bf c})$, where $\overline{c}_i(n) = c_i(n) m_i(n)$, for all $i \in \mathbb{N}$. Then with high probability
\begin{equation}\label{small:fraction:bound1}
 n^{-1}\mathcal{R}_n \geq  \E[R\psi_C (W^- z_0 )] >0.
\end{equation}
\end{theorem}

\begin{theorem}\label{prop:resilient:ext}
Assume that $(W^-,W^+,C)$ is such that there exists $z_0>0$ such that for any $0< z < z_0$
\begin{equation*}
\E [ W^+ W^- \P (\poi (z W^-)=C-1)\1_{\{C\geq 1 \}} ] <  1.
\end{equation*}
Then, for any sequence of ex post infections $\{M^{(j)} \}_{j \in \mathbb{N}}$ with $\lim_{j \rightarrow \infty}\P(M^{(j)}=0)=0$, let $ \mathcal{R}^{(j)}_{n}$ be the loss in relevance in $G_n (({\bf w^-}, {\bf w^+}, {\bf r}, {\bf c}^{j}  ))$, where $c_i^{(j)}=c_i (n) m_i^{(j)} (n)$ for all $i \in [n]$ and $j \in \mathbb{N}$. Then for any $\varepsilon >0$, there exists $j_{\varepsilon}$ such that for $j \geq j_{\varepsilon}$ with high probability
\begin{equation}\label{ext:resulzt:resilient}
n^{-1}\mathcal{R}^{(j)}_{n} \leq \varepsilon.
\end{equation}
\end{theorem}
For the changes necessary in the proof of Theorem~\ref{prop:resilient} one should observe that the fixpoints $\hat{z}_{M^{(j)}}$ used in the proof are not affected by the relevance variable $R$. Further by integrability of $R$ it follows that $\lim_{i\rightarrow \infty } \E [R \1 _{\{M^{(j)} =0\}}]=0$. Together with the continuity of the function $h(z):=\E [R \psi_C (z W^-)]$ it follows that $\lim_{i \rightarrow \infty } \E [R \psi_{C_{M^{(j)}}} (\hat{z}_{M^{(j)}} W^-)] =0 $ which implies (\ref{ext:resulzt:resilient}).

\bibliography{finance}
\bibliographystyle{abbrv}

\appendix

\section{Notes on the Derivative of $f$}\label{deriv:prop:f}

In this section we show that the requirements on $\E [ W^+ W^- \P (\poi (z W^-)=C-1)\1_{C\geq 1} ]$ in 2.\ of Theorem~\ref{thm:2} and in Theorem~\ref{prop:resilient} are satisfied if $f$ is differentiable with negative derivative around the fixpoint $\hat{z}$ or around $0$, respectively.
In fact, as the following two lemmas show, if $\E[W^+ W^-]< \infty $, the assumptions in Theorem~\ref{thm:2} and in Theorem~\ref{prop:resilient} could have equally be stated in terms of the derivative of $f$. However, when $\E[W^+ W^-]= \infty $ the function $f$ is not necessarily differentiable. 

\begin{lemma}\label{diff:f:second:moment}
If $\E[W^+ (W^-)^{1/2}]<\infty $ (respectively $\E[W^+ W^-]<\infty $) then $f$ is continuously differentiable in $(0,\infty)$ (respectively in $[0,\infty)$) and 
\begin{equation}
f' (z;(W^{-},W^{+}, C)) =\E [ W^+ W^- \P (\poi (z W^-)=C-1)\1_{C\geq 1} ]-1
\end{equation}
in the respective domains.
\end{lemma}
\begin{proof}
As for $l \ge 1$
\begin{eqnarray*}
\frac{\partial}{\partial z} \psi_l (x z)&=& \frac{\partial}{\partial z} \left( \sum^{\infty}_{r=l}   e^{-x z} \frac{(x z)^r}{r \!!} \right)= x e^{-x z} \frac{(x z)^{l-1}}{(l-1)!} = x \P (\poi ( z x)=l-1)
\end{eqnarray*}
and $\frac{\partial}{\partial z} \psi_0 (x z)=0$, formally differentiating $f$ below the integral sign provides 
\begin{equation}\label{diff:cand}
f' (z;(W^{-},W^{+}, C))= \E [ W^+ W^- \P (\poi ( z  W^-)=C-1)\1_{C\geq 1} ]-1.
\end{equation}
Since $\P (\poi ( x)=y)$ is maximized for $y=x$ and $\P (\poi ( x)=x)=(1/\sqrt{2 \pi x})(1+o(1))$ by Stirling's formula, it follows that 
\begin{equation}\label{unif:bound}
 W^+ W^- \P (\poi ( z  W^-)=C-1)\1_{C\geq 1} \leq  W^+ W^- \frac{1}{(z W^-)^{1/2}} D \leq (\hat{z}  - \delta_0)^{-1/2} W^+ ( W^-)^{1/2}D,
\end{equation}
for suitable chosen $D$ and for $z\geq \hat{z}-\delta_0$. By Assumption $\E[W^+ (W^-)^{1/2}]<\infty $ it follows that $\E [(\hat{z}  - \delta_0)^{1/2} W^+ ( W^-)^{1/2}]=(\hat{z}  - \delta_0)^{1/2}  \E [W^+ ( W^-)^{1/2}]<\infty$ and therefore $W^+ W^- \P (\poi ( z  W^-)=C-1)\1_{C\geq 1}$ is uniformly bounded by an integrable function for $z\in (\hat{z}-\delta_0,\infty )$. This justifies differentiation under the expectation and proves equality (\ref{diff:cand}) for $z\in (0,\infty)$ since $\delta_0$ can be chosen arbitrarily small. 
By continuity of $ W^+ W^- \P (\poi ( z  W^-)=C-1)\1_{C\geq 1}$ in $z$ and again the uniform bound in (\ref{unif:bound}) it follows by the Dominated Convergence Theorem that 
\begin{equation}
\E [ W^+ W^- \P (\poi ( z  W^-)=C-1)\1_{C\geq 1} ]
\end{equation}
is continuous and $f(z;(W^{-},W^{+}, C))$ continuously differentiable in $(0,\infty)$. In the case $\E[W^+ W^-]<\infty $, one can simply bound the left hand side of (\ref{unif:bound}) by $W^+ W^-$ for $z\in [0,\infty)$ as $\P (\poi ( z  W^-)=C-1)\1_{C\geq 1} \leq 1$ and apply the same arguments.
\end{proof}
\begin{lemma}\label{diff:smaller:1}
Let $\P (C=0)>0$ and $\hat{z}$ be the first positive solution of 
\begin{equation*}
f(z;(W^-,W^+,C)) = 0.
\end{equation*}
If $f$ is differentiable and $f'(z;(W^{-},W^{+}, C))<\kappa<0$ in some neighborhood of $\hat{z}$ or $f$ continuously differentiable in some neighborhood of $\hat{z}$ and $f'(\hat{z};(W^{-},W^{+}, C))<0$, then there exists $\delta_1>0$ such that
\begin{equation}\label{f:diff:bound}
\E [ W^+ W^- \P (\poi (z W^-)=C-1)\1_{C\geq 1} ] < \kappa <  1,\text{ for } z \in (\hat{z} - \delta_1, \hat{z}+ \delta_1 ).
\end{equation}

If, on the contrary, $\P (C=0)=0$ and further $f$ differentiable and $f'(z;(W^{-},W^{+}, C))<0$ for $z \in (0,\delta )$ and some $\delta >0$ or $f$ continuously differentiable in $[0,\delta )$ and $f'(0;(W^{-},W^{+}, C))<0$ for some $\delta > 0$ then 
\begin{equation}\label{f:diff:bound:0}
\E [ W^+ W^- \P (\poi (z W^-)=C-1)\1_{C\geq 1} ] < 1,\text{ for }z \in (0,\delta_1 )
\end{equation}
for some $\delta_1>0$.
\end{lemma}
\begin{proof}

We consider the case $\P(C=0)>0$. Note that if $f$ is continuously differentiable in some neighborhood of $\hat{z}$ and $f'(\hat{z};(W^{-},W^{+}, C))<0$, then necessarily $f'(z;(W^{-},W^{+}, C))<\kappa<0$ in some neighborhood of $\hat{z}$ and it is sufficient to prove (\ref{f:diff:bound}) under this Assumption. 

Since $f^{'}(z;(W^{-},W^{+}, C))< \kappa <0$ for $z\in (\hat{z} - \delta_1,\hat{z} + \delta_1 )$ and some $\delta_1$, it follows by the definition of $f$ that
\begin{equation*}
\lim_{h\rightarrow 0 }h^{-1}\left (\E [ W^{+} \psi_C( (z+h) W^{-} ) ] -\E [ W^{+} \psi_C( (z) W^{-} ) ]\right) <\kappa < 1
\end{equation*}
 for $z\in (\hat{z} - \delta_1,\hat{z} + \delta_1 )$. Furthermore, for $l \ge 1$
\begin{eqnarray*}
\frac{\partial}{\partial z} \psi_l (x z)&=& \frac{\partial}{\partial z} \left( \sum^{\infty}_{r=l}   e^{-x z} \frac{(x z)^r}{r \!!} \right)= x e^{-x z} \frac{(x z)^{l-1}}{(l-1)!} = x \P (\poi ( z x)=l-1)
\end{eqnarray*}
and $\frac{\partial}{\partial z} \psi_0 (x z)=0$. Therefore it is possible to write
\begin{equation*}
\E [ W^+ W^- \P (\poi ( z  W^-)=C-1)\1_{C\geq 1} ]=\E\left[\lim_{h\rightarrow 0} W^+ \left( \psi_C ((z+h)W^-)-\psi_C (z W^-)\right)\1_{C\geq 1}/h\right],
\end{equation*}
where $W^+ \left( \psi_C ((z+h)W^-)-\psi_C (z W^-)\right)\1_{C\geq 1}/h>0$ for every $h,z>0$. By Fatou's lemma, this allows us to conclude that for any sequence $\{h_i\}_{i\in \mathbb{N}}$, with $\lim_{i\rightarrow \infty} h_i=0$
\begin{eqnarray}\nonumber
&&\E[\liminf_{i \rightarrow \infty}W^+ \left( \psi_C ((z+h_i)W^-)-\psi_C (z W^-)\right)\1_{C\geq 1}/h_i]\\ \nonumber
&\leq &\liminf_{i \rightarrow \infty} \E[W^+ \left( \psi_C ((z+h_i)W^-)-\psi_C (z W^-)\right)\1_{C\geq 1}/h_i]\\ \nonumber
&=& \lim_{i \rightarrow \infty} \{ \left( \E[W^+ \psi_C ((z+h_i)W^-)]-\E[W^+ \psi_C (z W^-)] \right) /h_i \}\\ \nonumber
&=& f^{'}(z;(W^{-},W^{+}, C)) <\kappa  < 1,
\end{eqnarray}
for $z\in (\hat{z}-\delta_1,\hat{z}+\delta_1 )$. 

We consider the case $\P (C=0) =0$. There $f$ continuously differentiable in $[0,\delta_1)$ and $f'(0;(W^{-},W^{+}, C))<0$ implies that $f'(z;(W^{-},W^{+}, C))<0$ for $z \in (0,\delta_1 )$ and for some $\delta_1 >0$. With this observation the bound in (\ref{f:diff:bound:0}) can then be shown by essentially the same reasoning as in the case $\P (C=0)>0$.
\end{proof}
An interesting observation can be made for the case $\P (C=0)=0$. Assume there exists $z_0 >0$ such that $$f' (z)= \E [ W^+ W^- \P (\poi ( z \cdot W^-)=C-1)\1_{C\geq 1} ]-1,$$ for $z\in [0,z_0)$, that is $f$ is differentiable and $f'$ equals its candidate obtained by differentiation under the expectation sign, and further $f'$ continuous in this domain. This is especially the case if $\E [ W^+ W^-]< \infty$ as shown in Lemma~\ref{diff:f:second:moment}. Note that
\begin{equation}
\E [ W^+ W^- \P (\poi ( 0 \cdot W^-)=C-1)\1_{C\geq 1} ]=\E [ W^+ W^- \P (\poi ( 0  )=C-1)\1_{C = 1} ]
\end{equation}
as $\P (\poi ( 0)=c-1)= 0$ for $c> 1$. By continuity of $f'$,$$\E [ W^+ W^- \P (\poi ( 0  )=C-1)\1_{C = 1} ]<1$$ implies that $$\E [ W^+ W^- \P (\poi ( z \cdot W^- )=C-1)\1_{C\geq 1} ]<1$$ for $z$ in some neighborhood of $0$ and thus resilience depends only on vulnerable vertices. This shows that the functional $f$ has a delicate behaviour in $0$ if there are only few or no vulnerable vertices but the network is non-resilient. In fact, if $\E [ W^+ W^- \P (\poi ( 0  W^-)=C-1)\1_{C = 1} ]<1$ and the network is non-resilient, then often $\lim_{z \rightarrow 0 } f' (z) = \infty $.
\end{document}